\documentclass[a4paper,11pt]{article}
\usepackage{bm}
\usepackage{amsmath}
\usepackage{amsthm}
\usepackage{trfsigns}
\usepackage{wasysym}
\usepackage{amssymb}
\usepackage{amsfonts}
\usepackage{graphicx}
\usepackage[english]{babel}
\usepackage[utf8]{inputenc}
\usepackage[T1]{fontenc}
\usepackage{mathrsfs}
\usepackage{pdfsync}
\usepackage{enumerate}
\usepackage{version}
\usepackage{calc}
\usepackage{subfigure}

 \usepackage{pstricks,pst-math,pst-xkey}
 \usepackage{float}
 \usepackage{indentfirst}
\newtheorem{thm}{Theorem}
\newtheorem{cor}{Corollary}
\newtheorem{lem}{Lemma}

\newtheorem{prop}{Proposition}

\theoremstyle{remark}
\newtheorem{remark}{Remark}

\def\a{\alpha}
\def\b{\beta}
\def\l{\lambda}
\def\ve{\varepsilon}
\def\v{\varepsilon}

\def\dl{\delta}
\def\g{\gamma}

\def\W{\Omega}
\def\O{\Omega}
\def\d{\partial}
\def\p{\partial}

\def\w{\omega}
\def\o{\omega}
\def\R{\mathbb{R}}
\def\C{\mathbb{C}}
\def\S{\mathbb{S}}
\def\Div{\text{\rm  div}}
\def\deg{ \text{\rm deg}}
\def\n{\nabla}
\def\dist{\text{\rm dist}}
\def\tr{\text{\rm tr}}
\def\f{\varphi}
\def\Z{\mathbb{Z}}

\def\weak{\rightharpoonup}

\def\N{\mathbb{N}}

\author{Micka\"{e}l~\textsc{Dos\,Santos}
\\Universit\'{e} de Lyon, Universit\'{e} Lyon 1 , Institut Camille Jordan CNRS UMR 5208\\
43, boulevard du 11 novembre 1918, F-69622 Villeurbanne, France
\\{\tt dossantos@math.univ-lyon1.fr}
\and Oleksandr~\textsc{Misiats}\\Department of Mathematics, The Pennsylvania State University\\
University Park PA 16802, USA
\\{\tt misiats@math.psu.edu}}
\title{Ginzburg-Landau model with small pinning domains}
\begin{document}
\maketitle

\begin{abstract}
We consider a Ginzburg-Landau type energy with a piecewise constant pinning term $a$ in the potential $(a^2 - |u|^2)^2$. The function $a$ is different from 1 only on finitely many disjoint domains, called the {\it pinning domains}. These pinning domains model small impurities in a homogeneous superconductor and shrink to single points in the limit $\v\to0$; here, $\v$ is the inverse of the Ginzburg-Landau parameter. We study the energy minimization in a smooth simply connected domain $\Omega \subset \mathbb{C}$ with Dirichlet boundary condition $g$ on $\d \O$, with topological degree
${\rm deg}_{\d \O} (g)  = d >0$.  Our main result is that, for small $\v$, minimizers have $d$ distinct zeros (vortices) which are inside the pinning domains and they have a degree equal to $1$. The question of finding the locations of the pinning domains with vortices is reduced to a discrete minimization problem for a finite-dimensional functional of renormalized energy. We also find the position of the vortices inside the pinning domains and show that, asymptotically, this position is determined by {\it local renormalized energy} which does not depend on the external boundary conditions.
\end{abstract}


\tableofcontents
\section{Introduction and main results}

In this work we study the minimizers of the Ginzburg-Landau type functional
\begin{equation}\label{GL}
E_{\v, \dl} (u) = \frac{1}{2} \int_{\O}\left\{ |\nabla u|^2 + \frac{1}{2 \v^2} (a_\dl^2 - |u|^2)^2\right\},
\end{equation}
where $\W \subset \C$ is a bounded, smooth, simply connected domain, $\v$ is a positive parameter (the inverse of the Ginzburg-Landau parameter
$\kappa = 1/\v$), $\dl=\delta(\v) > 0$ is a geometric parameter and $u$ is a complex-valued map. In order to define the function $a_\dl$, we need to introduce the notion of a {\it pinning domain}. \\

Fix $M \in\N^*$ points $a_1, ..., a_M  \in \W$. Let $\w$ be an open subset
such that $\overline \w \subset B(0,1)$ and $0 \in \w$. For $1 \leq i \leq M$ and for all $\dl > 0$ denote $\w_\dl^i : = a_i + \dl \cdot\o$, \emph{i.e.} the set $\w$ scaled by $\dl$ and centered at $a_i$.
\\{\bf Definition.} The set $\w_\dl: = \cup_{i = 1}^M \w_\dl^i$ is called a {\it pinning domain}.

For example, if $\w = B(0,\frac{1}{2})$, then the pinning domain is $\w_\dl = \cup_{i = 1}^M B(a_i,\frac{\dl}{2})$.\\

We now define $a_\dl: \W \to \{b,1\}$, $b\in(0,1)$ as:
\begin{equation}\label{pin_term}
a_\dl(x) =
\begin{cases}
b& \text{if } x \in  \w_\dl \\
1& \text{if }  x \in \W \setminus \w_\dl
\end{cases}.
\end{equation}

The functionals of this type arise in models of superconductivity for composite superconductors. The experimental pictures suggest nearly 2D structure of
parallel vortex tubes (\cite{New}, Fig I.4). Therefore, the domain $\O$ can be viewed as a cross-section of a multifilamentary wire with a number of thin
superconducting filaments. Such multifilamentary wires are widely used in industry, including magnetic energy-storing devices, transformers and power
generators \cite{NATURE}, \cite{web}.

Another important practical issue in modeling superconductivity is to decrease the energy dissipation in superconductors.
Here, the dissipation occurs due to  currents associated with the motion of vortices (\cite{LinDu1}, \cite{BarSte}). This  dissipation as well the
thermomagnetic stability can be improved by {\it pinning}
(``fixing the positions'') of vortices. This, in turn, can be done
by introducing impurities or inclusions in the superconductor. In the functional (\ref{GL}) the set $\w_\dl$ models the set of
small impurities in a homogeneous superconductor. The size of the impurities in our model is characterized by the geometric parameter $\delta$ which
goes to zero together with the material parameter $\v$.
We assume henceforth that
\begin{equation}\tag{H}\label{7.MainHyp}
\frac{|\ln\delta(\v)|^3}{|\ln \v|}\to0.
\end{equation}
Essentially, this condition means that $\delta$ is much larger than $\v$ on the logarithmic scale.
For example, if $\v=2^{-j}$ and $\delta(\v)=2^{-k(j)}$, then \eqref{7.MainHyp} implies that $\displaystyle\frac{k(j)^3}{j}\to0$.

\noindent{\bf Notation. }In what follows:
\begin{enumerate}[$\bullet$]
\item We consider a sequence $\v_n\downarrow0$ and we write $\v$ instead of ${\v_n}$; the dependence of $\v$ on $n$ is implicit.
\item We simply write $\delta$ (instead of $\delta(\v)$); the dependence of $\delta$ on $\v$ is implicit.
\end{enumerate}
We study the minimization problem for the functional (\ref{GL}) in the class
\begin{equation}\label{class}
H^1_g:=\{u\in H^1(\O,\C)\,|\,\tr_{\p\O}u=g\},
\end{equation}
where $g\in C^\infty(\p\O,\S^1)$ is such that $\deg_{\p\O}(g)=d>0$. Recall that the degree (winding number) of $g$ is defined as
\[
\deg_{\p\O}(g):=\frac{1}{2\pi}\int_{\p\O}g\times\p_\tau g\,{\rm d}\tau.
\]
Here ``$\times$'' stands for the vectorial product in $\mathbb{C}$, \emph{i.e.} $z_1 \times z_2= {\rm Im}(\overline{z_1}z_2)$, $z_1,z_2\in\mathbb{C}$,
and $\p_\tau$ is the tangential derivative.
The degree is an integer, and the condition $\deg_{\p\O}(u) =d>0$, $u \in H^1(\O, \mathbb{C})$ implies that $u$ must have at least $d$ zeros (counting multiplicity)
inside $\O$. The properties of the topological degree can be found, \emph{e.g.}, in \cite{B1} or \cite{BeMi1}.\\

Minimization problems for Ginzburg-Landau type functionals have been extensively studied by a variety of authors. The pioneering work on modeling
Ginzburg-Landau vortices is the work of Bethuel, Brezis and H{\'e}lein \cite{BBH}. In this work the authors suggested to consider a simplified Ginzburg-Landau model (\ref{GL}) with $a \equiv 1$ in $\O$ (\emph{i.e.} without pinning term), in which the physical source of vortices, the external magnetic field, is modeled via a Dirichlet boundary condition with a positive degree on the boundary (\ref{class}). The analysis of full Ginzburg-Landau functional, with induced and applied magnetic fields, was later performed by Sandier and Serfaty in \cite{SS1}.\\

The functional (\ref{GL}) with non-constant $a(x)$ was proposed by Rubinstein in \cite{Rub} as a model of pinning vortices for Ginzburg-Landau minimizers. Shortly after, André and Shafrir \cite{ASPin} studied the asymptotics of minimizers for a smooth (say $C^1$) $a$. One of the first works to consider a discontinuous pinning term, which models a composite two-phase superconductor, was \cite{LM1}. In this work, a single inclusion  described by a pinning term independent of the parameter $\v$ was considered for a simplified Ginzburg-Landau functional with Dirichlet boundary condition $g$ on $\partial\O$. Namely the pinning term is
\begin{equation*}
a(x)=\begin{cases}1&\text{if }x\in\O\setminus\o\\b&\text{if }x\in\o\end{cases},
\end{equation*}
here $\o$  is a simply connected open set s.t. $\overline{\o}\subset\O$. The main objective of \cite{LM1} was to establish that the vortices are attracted (pinned) by the inclusion $\w$, and their location inside $\w$ can be obtained via minimization of certain finite-dimensional functional of renormalized energy.
Full Ginzburg-Landau model with discontinuous pinning term was later considered by Aydi and Kachmar \cite{AK1}. An $\v$-dependent but continuous pinning
term $a_\v(x)$ was studied by Aftalion, Sandier and Serfaty in \cite{ASS1}. The work \cite{AndBauPhi} studies the case of the smooth $a$ with finite number of isolated zeros, and in
\cite{AlaBro} the pinning term $a$ takes negative values in some regions of the domain $\W$. The other works related to Ginzburg-Landau functional with
pinning term include, \emph{e.g.},  \cite{LinDu1}, \cite{SigTin}.

In this work, we consider the minimization problem (\ref{GL})-(\ref{class}) with a discontinuous pinning term given by (\ref{pin_term}).
We prove that despite the fact that $a_\ve \to 1$ a.e. as $\ve \to 0$, \emph{i.e.} the pinning term disappears in the limit, the pinning domains $\w_\dl$ capture
the vortices of Ginzburg-Landau minimizers of (\ref{GL}) for small $\v$.

The main difficulty in the analysis of this problem stems in the fact that the \emph{a priori} Pohozhaev type estimate
$\|1-|v|^2\|^2_{L^2(\O)} \leq C \ve^2$ for the minimizer $v$ (on which the analysis in \cite{BBH} and \cite{LM1} is based) does not hold.
Therefore, we develop a different strategy of reducing the study of the minimizers of (\ref{GL}) to the analysis of $\S^1$-valued maps via the uniform estimates on the modulus of minimizers away from
the pinning domains (see Proposition \ref{P7.ConvOfBadDiscsTo0} below).

Following \cite{LM1}, let $U_{\v}$ be {\bf the} unique global minimizer of $E_\v$ in $H^1$ with $U_{\v} \equiv 1$ on $\d \O$. This $U_\ve$ satisfies $b \leq U_\v \leq 1$. For $v\in H^1_g$ we define
\[
F_\ve(v) = F_\ve(v, \O) := \frac{1}{2} \int_{\O} \left\{U_{\ve}^2 |\n v|^2 + \frac{1}{2 \ve^2} U_{\ve}^4(1-|v|^2)^2 \right\}\, {\rm d}x.
\]
Using the Substitution Lemma of \cite{LM1}, we have that for $v\in H^1_g$,
\begin{equation}\label{7.DecouplageEnergyLM}
E_\v(U_\v v)=E_\v(U_\v)+F_\v(v).
\end{equation}
From the decomposition \eqref{7.DecouplageEnergyLM}, we can reduce the minimization problem (\ref{GL})-(\ref{class}) to the minimization problem for $F_\ve$ in $H^1_g$, namely, the minimizer $v_\ve$ of $F_\ve$ in $H^1_g$ has the same vorticity structure as the original minimizer $u_\v$ of (\ref{GL})-(\ref{class}).

%
%
%
%

Depending on the relation between $M$  (number of inclusions), and $d$  (number of vortices), we distinguish two cases:\\
{\bf Case I:} $M \geq d$ (more inclusions than vortices), 
\\
{\bf Case II:} $M < d$ (more vortices than inclusions).

For example, we are going to show that for the minimizer $v_\ve$:
\begin{enumerate}[$\bullet$]
\item if $M=3$ and $d=2$ (Case I), we have two distinct inclusions containing exactly one zero each,
\item if $M=2$ and $d=3$ (Case II), we have one zero inside one inclusion and two distinct zeros inside the other inclusion.
\end{enumerate}
Generally speaking, outside a fixed neighborhood of $d'=\min\left\{d,M\right\}$ inclusions (centered at ${\bm a}=(a_{i_1},...,a_{i_{d'}})$), the minimizer $v_\v$ is almost an $\S^1$-valued map. Moreover, by minimality of ${v}_\v$, the selection of centers of inclusion containing its zeros and the distribution of degrees of $v_\v$ around the $a_i$'s are related to the minimization of the Bethuel-Brezis-Hélein renormalized energy $W_g$. In other words, we reduce the problem of finding vortices of the minimizers $v_\v$ to a two-step procedure. As the first step, we determine the inclusions with vortices, which is a discrete minimization problem for $W_g$ and is significantly simpler then the minimization of this renormalized energy functional over $\Omega^{d'}$. Secondly, we determine the locations of the zeros (vortices) locally inside each inclusion and show that their positions depend only on $b$, on the geometry of $\omega$ and on the relation between $d$ and $M$, but not on the external Dirichlet boundary condition $g$ (see Theorem \ref{T7.THM4} below).

Our main result in Case I is the following:
\begin{thm}\label{T7.THM1}
Assume that $M\geq d$. Let $v_\v$ be a minimizer of $F_\v$ in $H^1_g(\W)$. For any sequence $\v_n\downarrow0$, possibly after passing to a subsequence,
there are $d$ distinct points $\{a_{i_1}, ..., a_{i_d}\} \subset \{a_i, 1 \leq i \leq M \}$ and a
function $v^* \in H^1_{\rm loc}(\overline{\O} \setminus {\{a_{i_1}, ..., a_{i_d}\}}, \S^1)$
such that:
\begin{enumerate}
\item $v^*$ is a harmonic map, \emph{i.e.}
\begin{equation}\label{harmonic_map}
\begin{cases}
- \Delta v^* = v^* |\nabla v^*|^2 &\text{in } \O \setminus \{a_{i_1}, ..., a_{i_d}\}\\
v^* = g &\text{on } \d \O
\end{cases}.
\end{equation}
\item We have $v_{\v_n} \to v^*$ strongly in $H^1_{\rm loc}(\overline{\O} \setminus {\{a_{i_1}, ..., a_{i_d}\}})$ and $v_{\v_n} \to v^*$ in $C^\infty_{\rm loc}(\O \setminus {\{a_{1}, ..., a_{M} \}})$.\\
\item $v_{\v_n}$ has $d$ distinct vortices $x_1^n, ..., x_d^n$ such that $x_m^n$ is inside $\w_{\dl}^{i_m}$, $m = 1, ..., d$ and for small fixed $\rho$, $\deg_{\p B(x_i^n,\rho)}(v_{\v_n})=1$.
\item The following expansion holds
\begin{equation}\label{expansion}
F_{\ve}(v_\v) = \pi d b^2 |\ln \ve| + \pi (1-b^2) d |\ln \dl| + W_g((a_{i_1},1), ..., (a_{i_d},1))+ \tilde W   +o_\v(1).
\end{equation}
Here $\tilde W > 0$ is a local renormalized energy depending only on $d, b$ and $\o$. Moreover,  the $d$-subset $\{a_{i_1}, ..., a_{i_d}\}\subset\{a_1,...,a_M\}$ minimizes the Bethuel-Brezis-Hélein renormalized energy $W_g$ among the $d$-subsets of $\{a_1, ..., a_M\}$.
\end{enumerate}
\end{thm}
\begin{remark}
Here, $W_g$ denotes the renormalized energy given by Theorem I.7 in \cite{BBH} (with the degrees equal to $1$ and the boundary data $g$). Its definition is recalled in Section \ref{S7;MinBBHREnENErGy}.
\end{remark}
The main result in Case II is
\begin{thm}\label{T7.THM2}
Assume that $M<d$. Let $v_\v$ be a minimizer of $F_\v$ in $H^1_g(\W)$.  For any sequence $\v_n\downarrow0$, possibly after passing to a subsequence,
there is
$v^* \in H^1_{\rm loc}(\overline{\O} \setminus {\{a_1, ..., a_M\}}, \S^1)$ which satisfies \eqref{harmonic_map} in $ \O \setminus \{a_{1}, ..., a_{M}\}$, 
such that:
\begin{enumerate}
\item $v_{\v_n} \to v^*$ strongly in $H^1_{\rm loc}(\overline{\O} \setminus {\{a_{1}, ..., a_{M}\}})$ and $v_{\v_n} \to v^*$ in $C^\infty_{\rm loc}(\O \setminus {\{a_{1}, ..., a_{M}\}})$.\\
\item For $\rho>0$ small, $v_{\v_n}$ has exactly $d_i:=\deg_{\p B(a_i,\rho)}(v_{\v_n})$ zeros in $B(a_i,\rho)$. They are isolated, lie inside $\w_\dl^i$ and they have a degree equal to $1$.
\item \begin{equation}\label{constraint}
\left[\frac{d}{M}\right] \leq d_i \leq \left[\frac{d}{M}\right] + 1, \text{ where } \left[\frac{d}{M}\right] \text{ is the integer part of } \frac{d}{M}.
\end{equation}
Moreover, if $\frac{d}{M} = m_0 \in \mathbb{N}$, then $d_i \equiv m_0, \,1 \leq i \leq M$. Otherwise, the configuration $\{(a_1,d_1)$$,...,$$(a_M,d_M)\}$ minimizes the renormalized energy $W_g$ among the configurations \\
$\{(a_1,\tilde{d}_1),...,(a_M,\tilde{d}_M)\}$. Here $\{a_i\,|\, 1 \leq i \leq M\}$ are fixed and $\tilde d_i \in \Z$ are the subjects to the constraints (\ref{constraint}) and $\sum_{i=1}^M \tilde d_i = d$.
%
\item The following expansion holds when $\v\to0$
\begin{equation}\label{expansionII}
\inf_{H^1_g}F_{\ve} = \pi d b^2 |\ln \ve| + \pi (\sum_{i=1}^M d_i^2- d b^2)  |\ln \dl| + W_g\left(\{{\bm a,\bm d}\}\right) +  \tilde W  + o_\v(1).
\end{equation}
Here, $\{{\bm a},{\bm d}\}=\{(a_1,d_1),...,(a_M,d_M)\}$ is a configuration given by the previous assertion and $\tilde{W}$ is local renormalized energy which depends only on $\o,b,d$ and $M$.
\end{enumerate}
\end{thm}
In both cases, we prove that the asymptotic location of the vortices inside a pinning domain depends only on $b$, $\o$ and on the number of zeros inside the inclusion (see Theorem \ref{T7.THM4}): this location is independent of the boundary data $g$ on $\p\O$.



\section{Main tools}\label{S7.Maintools}
In this section we establish:
\begin{enumerate}[$\bullet$]
\item Estimates for $U_\v$,
\item Upper bounds for the energy of minimizers in Case I and Case II, 
\item An $\eta$-ellipticity estimate for minimizers.
\end{enumerate}
\subsection{Properties of $U_\ve$}
\begin{prop} [Maximum principle for $U_\v$, \cite{LM1} Proposition 1]

The special solution $U_\ve$ satisfies $b \leq U_\ve \leq 1$ in $\Omega$.
\end{prop}
\begin{prop}\label{P7.UepsCloseToaeps}
There are $C,c>0$ (independent of $\v$) s.t. for any $R >0$ we have
\begin{equation}\label{7.UepsCloseToaeps}
|a_\v-U_\v|\leq C{\rm e}^{-\frac{c R}{\v}}\text{ in }V_{R}:=\{x\in\O\,|\,\dist(x,\p\omega_\delta)\geq R\},
\end{equation}
\begin{equation}\label{P7.GradUepsCloseToaeps}
|\n U_\v|\leq \frac{C{\rm e}^{-\frac{ c R}{\v}}}{\v}\text{ in } V_{R}.
\end{equation}
\end{prop}
The proof of the Proposition \ref{P7.UepsCloseToaeps} is presented in the Appendix \ref{S.AppendixEstimateSpecialSol}.
\subsection{Upper Bounds}
\begin{prop}\label{P7.UpperboundAuxPb}Let $\displaystyle\xi=\frac{\v}{\delta}$.
\begin{enumerate}\item Upper bound in Case I: $M\geq d$

 There is a constant $C$ depending only on $g,\o$ and $\O$ s.t. we have
\begin{equation}\label{7.UpperboundAuxPb}
\inf_{H^1_{g}(\O)} F_\v (\cdot, \O) \leq \pi d b^2|\ln \xi|+\pi d |\ln\delta|+C.
\end{equation}
\item  Upper bound in Case II: $M<d$

There is a constant $C$ depending only on $g,\o$ and $\O$ s.t. for all $d_1,...,d_M\in\N$ s.t. $\sum d_i=d$ we have
\begin{equation}\label{7.UpperboundAuxPbCaseII}
\inf_{H^1_{g}(\O)} F_\v (\cdot, \O) \leq \pi d b^2|\ln \xi|+\pi \sum_i d_i^2 |\ln\delta|+C.
\end{equation}
\end{enumerate}
\end{prop}
The proof of Proposition \ref{P7.UpperboundAuxPb} is given in Appendix \ref{S.ProofUpperBound}.

\subsection{Identifying bad discs}
\begin{lem}\label{L7.BadDiscsLemma}
Let $g_\v,g_0\in C^\infty(\p\O,\C)$ be s.t. $0\leq1-|g_\v|\leq \v$ and $g_\v\to g_0$ in $C^1(\p\O)$. Let also $\alpha_\v,\beta_\v\in L^\infty(\O,[b,1])$.

Consider the weighted Ginzburg-Landau functional
\[
F^w_\v(v)=\frac{1}{2}\int_\O\left\{\alpha_\v|\n v|^2+\frac{\beta_\v}{\v^2}(1-|v|^2)^2\right\}.
\]

Denote $v_\v$ a minimizer of $F^w_\v$ in $H^1_{g_\v}$. Then the following results hold:
\begin{enumerate}
\item Let  $\chi=\chi_\v\in(0,1)$ be s.t. $\chi\to0$. There are $\v_0>0$, $C>0$ and $C_1>0$ depending only on $b,\chi,\O,\|g_0\|_{C^1(\p\O)}$ s.t for $\v<\v_0$, if
\[
F^w_\v(v_\v,B(x,\v^{1/4})\cap\O)\leq \chi^2|\ln\v|-C_1,
\]
then
\[
|v_\v|\geq1-C\chi\text{ in }B(x,\v^{1/2})\cap\O.
\]
\item Let $\mu\in(0,1)$. Then there are $\v_0,C>0$ depending only on $b,\mu,\O,\|g_0\|_{C^1(\p\O)}$  s.t. for $\v<\v_0$, if
\[
F^w_\v(v_\v,B(x,\v^{1/4})\cap\O)\leq C|\ln\v|,
\]
then
\[
|v_\v|\geq\mu\text{ in }B(x,\v^{1/2})\cap\O.
\]
\end{enumerate}
\end{lem}
Lemma \ref{L7.BadDiscsLemma} is proved in Appendix \ref{S.ProofEtaEllipticityLemma}.

\section{A model problem: one inclusion}\label{S7.ProofTheorem3}
By combining the results of Section \ref{S7.Maintools}, the proofs of both Theorem \ref{T7.THM1} and Theorem \ref{T7.THM2} are based on the analysis of two distinct problems:
\begin{enumerate}
\item A minimization problem of the Dirichlet functional among $\S^1$-valued map defined on a perforated domain.
\item The study of the minimizers $v_\v$ around an inclusion.
\end{enumerate}
This section focuses on the second problem. More precisely, we fix $\rho>0$ and study the minimization problem of $F_\v(\cdot,B(a_i,\rho))$ with variable boundary conditions.

Fix $\rho > 0$ and let $f_\v$, $f_0 \in C^{\infty}(\d B(0,\rho))$ be s.t. $f_0$ is $\S^1$-valued and s.t.
\begin{equation}\label{A1}
\|f_\v - f_0\|_{C^1(\d B(0, \rho))} \to0
\end{equation}
and
\begin{equation}\label{A2}
\||f_\v|-1\|_{L^2(\p B(0,\rho))} \leq C \v^2.
\end{equation}
Assume also that $\deg_{ \d B(0,\rho)}(f_\v) = \deg_{ \d B(0,\rho)}(f_0) = d_0 > 0$.


For $i\in\{1,...,M\}$ consider the minimization problem
\begin{equation}\label{modelGL}
F_\ve(v, B(a_i, \rho)) := \frac{1}{2} \int_{B(a_i, \rho)} \left\{U_{\ve}^2 |\n v|^2 + \frac{1}{2 \ve^2} U_{\ve}^4(1-|v|^2)^2 \right\}\, {\rm d}x
\end{equation}
in the class
\begin{equation}\label{modelclass}
H_{f_\v,i}^1 := \{v \in H^1(B(a_i, \rho),\C)\,|\,\tr_{\p B(a_i, \rho)} v(x)= f_\ve(x-a_i)\}.
\end{equation}

Without loss of generality assume $a_i=0$. Let $v_\ve$ be a minimizer of (\ref{modelGL}) in (\ref{modelclass}). Performing the change of variables $\hat x  = \frac{x}{\dl}$ in
(\ref{modelGL}), we have
\begin{equation}\label{variable_change}
F_\ve(v_\ve, B(0, \rho)) = \hat F_{\xi}(\hat v_\v, B(0, \frac{\rho}{\dl})) := \frac{1}{2} \int_{B(0, \frac{\rho}{\dl})} \left\{{\hat U}_{\ve}^2 |\n \hat v|^2 + \frac{1}{2 \xi^2} {\hat U}_{\ve}^4(1-|\hat v|^2)^2 \right\}\, {\rm d} \hat x.
\end{equation}
Here, for a map $w\in H^1(B(0,\rho))$, we denote $\hat w (\hat x) := w( \delta \hat x)$ and $\xi =\displaystyle \frac{\ve}{\dl}$.
The class (\ref{modelclass}) under this change of variables becomes
\begin{equation}\label{modelclassblownup}
\hat H_{f_\v}^1 := \left\{\hat v \in H^1(B(0, \frac{\rho}{\dl}),\C)\,|\,\tr_{\p B(0, \frac{\rho}{\dl})} \hat v( \cdot) = f_\ve(\dl \cdot)\right\}.
\end{equation}
Note that the above rescaling enables us to fix the pinning domain independently of $\v$.

{The asymptotic behavior of $\hat{v}_\v$ will be obtained in several steps:
\begin{enumerate}[$\bullet$]
\item We first establish a bound for $|\hat{v}_\v|$ (Proposition \ref{P7.ConvOfBadDiscsTo0}). This bound will allow us to localize (roughly) the vortices of $v_\v$ near the inclusion.
\item We next establish sharp energy estimates (Proposition \ref{P7.LocalizationOfTheEnergy}) and use them to obtain the uniform convergence of solutions away from the inclusion (Proposition \ref{P7.SmoothBound} and Corollary \ref{C7.SmoothConv}). We establish the strong $H^1$ convergence of solutions away from the "vortices" (Proposition \ref{P7.AsymptoticResultatNonZeroDeg}) and derive the equation satisfied by the limiting map (Proposition \ref{P7.equationforphistar}).
\item The last step is the location and quantization of the vorticity: for small $\v$, the minimizers admits exactly $d_0$ zeros, and all the zeros lie in the inclusion and have a degree equal to $1$ (Propositions \ref{P7.AsymptoticResultatNonZeroDeg} and \ref{P7.uniq0}).
\end{enumerate}
}
Following the same lines as for Proposition \ref{P7.UpperboundAuxPb}, one may prove
\begin{prop}\label{P7.UpperBndTh3} Let $\hat v_\ve$ be a minimizer of $\hat F_\xi$ in (\ref{modelclass}). Then there is a constant $C$ independent of $\v$ s.t. we have
\begin{equation}\label{Uppperbndfor3}
 F_\v(v_\v, B(0, \rho))=\hat F_\xi (\hat v_\ve, B(0, \frac{\rho}{\delta})) \leq \pi d_0 b^2|\ln \xi|+\pi d_0^2 |\ln\delta|+C.
\end{equation}
\end{prop}

\subsection{Uniform convergence of $|\hat v_\v|$ to $1$ away from inclusions}

\begin{prop}\label{P7.ConvOfBadDiscsTo0}
Let  $K\subset\R^2$ be a compact set such that $\o \subset K$ and $\dist(\d K,\o)>0$. Then there is $C>0$ independent of $\v$ s.t. for sufficiently small $\v$ we have
\[
|\hat v_\v|\geq1-C{|\ln\v|}^{-1/3} \text{ in }B_{\frac{\rho}{\dl}}\setminus K.
\]
\end{prop}
\begin{proof}
Using Lemma \ref{L7.BadDiscsLemma} with $\displaystyle\chi=|\ln \ve|^{-1/3}$, we find that there exist $C,C_1>0$ s.t. for $\v>0$ small, if $\displaystyle F_\v(v_\v,B(x,\v^{{1/4}}))< |\ln \ve|^{\frac{1}{3}}-C_1$ then $|v_\v|\geq 1-C\chi$ in $B(x,\v^{1/2})$.

We argue by contradiction. Assume that there exists a compact $K$ containing $\w$ s.t. $\dist(\p K, \o)>0$ and s.t., up to a subsequence, there is a sequence of points $\hat x_{\ve} \in B(0, \frac{\rho}{\delta}) \setminus K$ s.t. $|\hat v_{\ve}(\hat x_{\ve})| < 1 - C |\ln \ve|^{-1/3}$ with $C$ given by Lemma \ref{L7.BadDiscsLemma}. Note that $\hat x_{\ve} \in B(0, \frac{\rho}{\delta}) \setminus K$ corresponds to $x_{\ve} \in B(0, \rho) \setminus (\dl\cdot K)$.
From Lemma \ref{L7.BadDiscsLemma}  and Proposition \ref{P7.UepsCloseToaeps}
\begin{equation}\label{7.ContradictionForUnifConv}
\frac{1}{2}\int_{B(x_\v, {\v^{1/4}})} \left\{|\n v_{\v}|^2+\frac{1}{2\v^2} (1-|v_{\v}|^2)^2\right\} \geq |\ln \v|^{1/3} - \mathcal{O}(1).
\end{equation}
We claim that due to the conditions (\ref{A1}), (\ref{A2}), we may extend $v_{\v}$ (keeping the same notation for the extension) to a smooth map, still denoted $v_{\v}$, s.t.

\begin{equation}\label{7.ExtensionPropertyAucPron}
\begin{cases}
v_{\v}(x)=x^{d_0}/|x|^{d_0}\text{ in }B(0, 3\rho) \setminus \overline{B(0, 2\rho)}
\\
\displaystyle\int_{B(0, 3\rho) \setminus \overline{B(0, \rho)}}(1-|v_{\v}|^2)^2 \leq C \v^2
\\
|\n v_{\v}|\leq C\text{ with }C>0\text{ is }\text{independent of $\v$}
\end{cases}.
\end{equation}
To make the above extension explicit, choose $\zeta\in C^\infty(\R^+,[0,1])$ s.t. $\zeta\equiv0$ in $[0,\rho]$ and $\zeta\equiv1$ in $[2\rho,3\rho]$ and take
\[
v_{\v}(s{\rm e}^{\imath\theta})=\left[\zeta(s)+(1-\zeta(s))|f_\v(\rho{\rm e}^{\imath\theta})|\right]{\rm e}^{\imath\left[d_0\theta+(1-\zeta(s))\phi_{\v}(\rho{\rm e}^{\imath\theta})\right]}.
\]
Here $x=s{\rm e}^{\imath\theta},\,s>0\text{ and }\phi_{\v}\in C^\infty(\p B(0,\rho),\R)\text{ is s.t. }f_\v (\rho{\rm e}^{\imath\theta}) =|f_\v|{\rm e}^{\imath\left(d_0\theta+\phi_{\v}\right)}$.
Consequently, as follows from \eqref{7.UpperboundAuxPb} and \eqref{7.ExtensionPropertyAucPron}, this map satisfies
\[
\frac{1}{2}\int_{B(0, 3\rho)}{\left\{|\n v_{\v}|^2+\frac{1}{2\v^2}(1-|v_{\v}|^2)^2\right\}}\leq  C |\ln\v|.
\]
Therefore, the map $v_{\v}$ in $B(0, 3\rho)$ fulfills the conditions of Theorem 4.1 in \cite{SS1}. This theorem guarantees that:
\begin{enumerate}[$\bullet$]
\item we may cover the set $\{x\in B(0,3\rho-\v/b) \, | \, |v_{\v}(x)| < 1 - ({\v}/{b})^{1/8}\}$ with a finite collection of disjoint balls $\mathcal{B}^\v := \{B_j^\v\}$;
\item the radius of $\mathcal{B}^\v$, ${\rm rad}(\mathcal{B}^\v)$, which is defined as the sum of the radii of the balls $B_j^\v$, ${\rm rad}(\mathcal{B}^\v):=\sum_j{\rm rad}(B_j^\v)$, satisfies ${\rm rad}(\mathcal{B}^\v) \leq  10^{-2}\dl \cdot{\rm dist}(\w, \d K)$;
\item denoting $d_j=\deg_{\p  B^\v_j}(v_{\v})$ if $B^\v_j\subset B(0,3\rho-{\v}/{b})$ and $d_j=0$ otherwise;
\end{enumerate}
we have
\begin{equation}\label{7.LowerBoundOnSandierSerfatyCovering1}
\frac{1}{2}\int_{\mathcal{B}^\v}{\left\{|\n v_{\v}|^2+\frac{b^2}{2\v^2}(1-|v_{\v}|^2)^2\right\}} \geq \pi \sum_j|d_j| \ln\frac{\delta}{\v}-C.
\end{equation}
Note that, by the construction of $v_{\v}$ in $B(0,3\rho)\setminus\overline{B(0,\rho)}$, if we have $\deg_{\p  B^\v_j}(v_{\v})\neq0$ then $B_j^\v\subset B(0,5\rho/2)$. Thus $d_j=\deg_{\p  B^\v_j}(v_{\v})$ for all $j$.

In order to obtain a lower bound for $F_{\v}$ we use the identity
\begin{eqnarray}\label{identity}
F_{\v}(v_{\v}, B(x_{\v}, \v^{1/4}) \cup \mathcal{B}^\v)  &=& \frac{b^2}{2}\int_{B(x_{\v}, \v^{1/4}) \cup \mathcal{B}^\v}\left\{ |\n v_{\v}|^2+\frac{b^2}{2\v^2}(1-|v_{\v}|^2)^2\right\} \\
\nonumber &&+\frac{1}{2}\int_{B(x_{\v}, \v^{1/4}) \cup \mathcal{B}^\v}\Big\{ (U_{\v}^2-b^2) |\n v_{\v}|^2
\\\nonumber&&\phantom{aaaassssssssss}+\frac{1}{2\v^2}(U_{\v}^4-b^4)(1-|v_{\v}|^2)^2\Big\}.
\end{eqnarray}
The first integral in \eqref{identity} is estimate via \eqref{7.LowerBoundOnSandierSerfatyCovering1}:
\begin{eqnarray}\nonumber
\frac{b^2}{2}\int_{B(x_{\v}, \v^{1/4}) \cup \mathcal{B}^\v}\left\{ |\n v_{\v}|^2+\frac{b^2}{2\v^2}(1-|v_{\v}|^2)^2\right\}& \geq&
\pi b^2 \sum_j|\deg_{\p B_j}(v_{\v})| \ln \frac{\delta}{\v}-C
\\\label{lb1}
& \geq& \pi b^2 d_0 \ln\frac{\delta}{\v} - C_0.
\end{eqnarray}
By combining \eqref{7.ContradictionForUnifConv} and Proposition \ref{P7.UepsCloseToaeps}, we have for small $\v$
\begin{eqnarray}\nonumber
\frac{1}{2}\int_{B(x_{\v}, \v^{1/4}) \cup \mathcal{B}^\v} \left\{(U_{\v}^2-b^2) |\n v_{\v}|^2+\frac{1}{2\v^2}(U_{\v}^4-b^4)(1-|v_{\v}|^2)^2\right\} \\
\label{lb2} \geq \frac{1-b^2}{2} \int_{B(x_{\v}, \v^{1/4})}\left\{|\n v_{\v}|^2+\frac{1}{2\v^2}(1-|v_{\v}|^2)^2\right\} -C\geq (1-b^2) |\ln \v|^{1/3} - C';
\end{eqnarray}
here we rely on the assumption \eqref{7.MainHyp} on the behavior of $\delta(\v)$ as $\v\to0$.

Substituting the bounds (\ref{lb1}) and (\ref{lb2}) in (\ref{identity}) we obtain a contradiction with (\ref{7.UpperboundAuxPb}). This completes the proof of Proposition \ref{P7.ConvOfBadDiscsTo0}.
\end{proof}

\subsection{Distribution of Energy in $B(0, \frac{\rho}{\dl})$}

\begin{prop}\label{P7.LocalizationOfTheEnergy}
The following estimates hold:
\begin{equation}\label{7.EnergyOutsideTheHole}
\frac{1}{2}\int_{B(0, \rho/\dl)\setminus \overline{B(0,1)}}{\hat U_\v^2|\n\hat v_\v|^2}=\pi d_0^2|\ln\delta|+\mathcal{O}(1),
\end{equation}
and (recall that $\xi=\displaystyle\frac{\v}{\delta}$)
\begin{equation}\label{7.EnergyOfTheHole}
\hat F_\xi(\hat v_\v,B(0,1))=\pi d_0 b^2|\ln\xi|+\mathcal{O}(1).
\end{equation}


\end{prop}
\begin{proof}We start by proving that
\begin{equation}\label{7.EnergyOfTheHoleBis00}
\hat F_\xi(\hat v_\v,B(0,1)) \geq \pi d_0 b^2|\ln\xi|-\mathcal{O}(1).
\end{equation}
As before, we use Theorem 4.1 in \cite{SS1}: for $0<r<r_0:=10^{-2}\cdot\dist(\o,\p B(0,1))$, there are $C>0$ and a finite covering by disjoint balls $B^\v_1,...,B^\v_N$ (with the sum of radii at most $r$) of the set $\{\hat x\in B(0,1-\xi/b) \,|\,1-|\hat v_{\v}(\hat x)|\geq({\xi}/{b})^{1/8}\}$ s.t.
\begin{equation}\label{7.EnergyOfTheHoleAux00}
\frac{1}{2}\int_{\cup_jB^\v_j}{\left\{|\n \hat v_{\v}|^2+\frac{b^2}{2\xi^2}(1-|\hat v_{\v}|^2)^2\right\}}\geq\pi D|\ln\xi|-C,
\end{equation}
with $D=\sum_{j}|d_j|$ and
\[
d_j=\begin{cases}\deg_{\p B^\v_j}(\hat v_{\v})&\text{if }B^\v_j\subset B(0, 1- {\xi}/{b}) \\0&\text{otherwise}\end{cases}.
\]
From Proposition \ref{P7.ConvOfBadDiscsTo0}, for $\v$ small, if $\deg_{\p B^\v_j}(\hat v_{\v})\neq0$ then $B^\v_j\subset B(0,1-r_0) \subset B(0, 1 - \xi/b)$. It
follows that $D\geq d_0$ and then \eqref{7.EnergyOfTheHoleBis00} is a direct consequence of \eqref{7.EnergyOfTheHoleAux00} and the bound $\hat{U}_\v\geq b$.

We next prove that there is $C > 0$ s.t.
\begin{equation}\label{7.EnergyOutsideTheHoleBis0}
\frac{1}{2}\int_{B(0, \frac{\rho}{\dl})\setminus \overline{B(0,1)}}{\hat U_\v^2|\n\hat v_\v|^2}\geq\pi d_0^2|\ln\delta| - C.
\end{equation}
By Proposition \ref{P7.ConvOfBadDiscsTo0}, $|v_{\v}|\geq1/2$  in $B(0, \frac{\rho}{\delta}) \setminus \overline{B(0, 1)}$, therefore, $\hat w_{\v} :=  \frac{\hat v_{\v}}{|\hat v_{\v}|}$ is well-defined in this domain.
Observe that
\begin{equation}\label{for_S10}
\frac{1}{2} \int_{B(0, \frac{\rho}{\delta}) \setminus \overline{B(0, 1)}} |\nabla \hat w_{\v}|^2 \geq \frac{1}{2} \int_{B(0, \frac{\rho}{\delta}) \setminus \overline{B(0, 1)}} \left|\nabla \frac{z^{d_0}}{|z|^{d_0}}\right|^2=
 \pi d_0^2 \ln \frac{\rho}{\delta} .
\end{equation}
We claim that (\ref{7.EnergyOutsideTheHoleBis0}) holds with $C = \pi d_0^2 |\ln {\rho}| + 1$ (for small $\v$). By contradiction, assume (\ref{7.EnergyOutsideTheHoleBis0}) does not hold.
Then, up to a subsequence, we have
\begin{equation}\label{assumption0}
\frac{1}{2}\int_{B(0,\frac{\rho}{\delta})\setminus \overline{B(0,1)}}{\hat U_{\v}^2|\n\hat v_{\v}|^2} < \pi d_0^2\ln\frac{\rho}{\delta}- 1.
\end{equation}
On the other hand, we have
\[
|\nabla \hat v_{\v}|^2 = |\hat v_{\v}|^2 |\nabla \hat w_{\v}|^2 + |\nabla |\hat v_{\v}||^2
\]
and therefore
\begin{equation}\label{aux0}
\int_{B(0, \frac{\rho}{\delta}) \setminus \overline{B(0, 1)}} |\nabla \hat v_{\v}|^2 \geq \int_{B(0, \frac{\rho}{\delta}) \setminus \overline{B(0, 1)}} |\nabla \hat w_{\v}|^2  - \int_{B(0, \frac{\rho}{\delta}) \setminus \overline{B(0, 1)}} (1-|\hat v_{\v}|^2) |\nabla \hat w_{\v}|^2.
\end{equation}
Since $|\hat v_{\v}| \geq \frac{1}{2}$ in $B(0, \frac{\rho}{\delta}) \setminus \overline{B(0, 1)}$ we have $|\nabla \hat w_{\v}| \leq 2 |\nabla \hat v_{\v}|$.
Therefore, by (\ref{assumption0}), Proposition \ref{P7.ConvOfBadDiscsTo0} and \eqref{7.MainHyp} we estimate the last term in (\ref{aux0}):
\begin{equation}\label{last_term0}
\int_{B(0, \frac{\rho}{\delta}) \setminus \overline{B(0, 1)}} (1-|\hat v_{\v}|^2) |\nabla \hat w_{\v}|^2 \leq
C_2 |\ln {\v}|^{-\frac{1}{3}} \int_{B(0, \frac{\rho}{\delta})\setminus \overline{B(0, 1)}}|\nabla \hat v_{\v}|^2 \leq C_3 \frac{|\ln \delta|}{|\ln {\v}|^{\frac{1}{3}}} \to 0.
\end{equation}
Combining (\ref{for_S10}), (\ref{aux0}) and (\ref{last_term0}), we find that
\[
\int_{B(0, \frac{\rho}{\delta}) \setminus \overline{B(0, 1)}} |\nabla \hat v_{\v}|^2 \geq \pi d_0^2 \ln \frac{\rho}{\delta} - o_{\v }(1).
\]
Since $|\hat U_{\v}-1|\leq C\xi^4$ in $B_{\frac{\rho}{\delta}}\setminus \overline{B(0,1)}$ (see Proposition \ref{P7.UepsCloseToaeps}), we obtain a
contradiction with (\ref{assumption0}), and (\ref{7.EnergyOutsideTheHoleBis0}) follows. Comparing the lower bounds (\ref{7.EnergyOfTheHoleBis00}) and (\ref{7.EnergyOutsideTheHoleBis0}) with the upper bound in Proposition \ref{P7.UpperBndTh3}, the Proposition \ref{P7.LocalizationOfTheEnergy} follows.
\end{proof}

Using exactly the same techniques as in the proof of Proposition \ref{P7.LocalizationOfTheEnergy}, one may easily prove the following estimate.
\begin{cor}\label{C7.BoundEnergyU}
For any $R_2 > R_1 \geq 1$
\begin{equation*}
F_\xi(\hat v_\v, B(0, R_2) \setminus \overline{B(0, R_1)}) = \mathcal{O}(1).
\end{equation*}
\end{cor}
\subsection{Convergence in $C^\infty(K)$ for a compact $K$ s.t. $K\cap\overline{\o}=\emptyset$}
\begin{prop}\label{P7.SmoothBound}
Let $K\subset\R^2\setminus\overline{\o}$ be a smooth compact set. Then we have
\begin{equation}\label{7.BoundInCKaKa}
\hat v_\v\text{ is bounded in }C^k(K)\text{ for all }k\geq0
\end{equation}
and there is $C_K>0$ s.t.
\begin{equation}\label{5.EstMofUsiMirArg}
|\hat{v}_\v|\geq1-C_K \xi^2\text{ in }K.
\end{equation}
\end{prop}
\begin{proof}
From Proposition \ref{P7.UepsCloseToaeps}
\begin{equation}\label{7.BoundEnergyU}
E_\xi(\hat U_\v,K)=\frac{1}{2}\int_{K}{|\n \hat U_\v|^2+\frac{1}{2\xi^2}(1-\hat U_\v^2)^2}=\mathcal{O}(1).
\end{equation}
As in \cite{LM1}, the following expansion holds
\begin{equation}\label{7.ExpandingLocalEnergy}
E_\xi(\hat U_\v \hat v_\v,K)=E_\xi(\hat U_\v ,K)+\hat F_\xi(\hat v_\v,K)+\int_{\p K}{(|\hat v_\v|^2-1)\hat U_\v\p_\nu \hat U_\v}.
\end{equation}
Using \eqref{P7.GradUepsCloseToaeps}, we have
\[
\int_{\p K}{(|\hat v_\v|^2-1)\hat U_\v\p_\nu \hat U_\v}=o_\v(1).
\]
With \eqref{7.BoundEnergyU} and \eqref{7.ExpandingLocalEnergy}, we conclude that $E_\xi(\hat U_\v \hat v_\v,K)=\mathcal{O}(1)$.
Since $\hat U_\v$ and $\hat U_\v \hat v_\v$ satisfy the Ginzburg-Landau equation $-\Delta u = \frac{1}{\xi^2}u(1-|u|^2)$ in $K$, as well as $|\hat U_\v| \leq 1$
and $|\hat U_\v \hat v_\v| \leq 1$. Theorem 1 in \cite{M1} implies that
\[
\hat U_\v\text{ and }\hat U_\v\hat v_\v\text{ are bounded in }C^k(K)\text{ for all }k\geq 0.
\]
It follows that $\hat v_\v\text{ is bounded in }C^k(K)\text{ for each }k\geq0$. On the other hand, using the fact that $\hat{v}_\v$ is bounded in $C^k(K)$ together with the equation of $\hat{v}_\v$, we find that $1 - |\hat v_\v|^2 \leq C_K \xi^2$ in $K$.
\end{proof}
\begin{cor}\label{C7.SmoothConv}
For $K\subset\R^2\setminus\overline{\o}$, up to a subsequence, there is some $v_0\in C^\infty(K,\S^1)$ s.t. $\hat v_{\v}\to v_0$ in $C^\infty(K)$.
\end{cor}
We are now in position to bound the potential part of the energy.
\begin{cor}\label{C7.Pohozaev_bound}
There exists $C>0$ independent of $\v$ s.t.
\begin{equation}\label{Pohozaev_bound}
\frac{1}{\v^2} \int_{B(0,\rho)} (1-| v_{\v}|^2)^2=\frac{1}{\xi^2} \int_{B(0,\rho/\delta)} (1-|\hat v_{\v}|^2)^2 \leq C.
\end{equation}
\end{cor}
\begin{proof}
Note that from Propositions \ref{P7.UpperBndTh3} and \ref{P7.LocalizationOfTheEnergy}, we find that there is $C>0$ s.t.

\[
\frac{1}{\xi^2} \int_{B(0,\rho/\delta)\setminus\overline{B(0,1)}} (1-|\hat v_{\v}|^2)^2\leq C.
\]
Thus it remains to prove the estimate in $B(0,1)$ for small $\v$. Using \eqref{7.BoundInCKaKa}, ${\rm tr}_{\d B(0,1)} \hat v_\v$ is bounded in $C^{1}(\d B(0,1))$ and $1 - |\hat v_\v|^2 \leq C \xi^2$ on
$\d B(0,1)$ (for small $\v$). These properties, allow us to construct a smooth extension $\tilde v_\v$ of ${\rm tr}_{\d B(0,1)} \hat v_\v$ into $B(0,2) \setminus \overline{B(0,1)}$, s.t. $h=\tr_{\p B(0,2)}\tilde v_\v $ is $\S^1$-valued and independent of $\v$, $1 - |\tilde v_\v|^2 \leq C \xi^2$ in $B(0,2) \setminus \overline{B(0,1)}$ and
\begin{equation}\label{bound_in_ring}
\int_{B(0,2) \setminus \overline{B(0,1)}}\left\{|\nabla \tilde v_\v|^2 + \frac{1}{2 \xi^2} (1 - |\tilde v_\v|^2)^2 \right\}\leq C_0.
\end{equation}
(For example, this construction is performed by mimicking \eqref{7.ExtensionPropertyAucPron})

Define $w_\v$ as $w_\v = \hat v_\v$ in $B(0,1)$ and $w_\v = \tilde v_\v$ in $B(0,2) \setminus \overline{B(0,1)}$. Clearly, $w_\v \in H^1_h(B(0,2))$, $w_\v$ is bounded in
$L^2(B(0,2))$ and, thanks to Proposition \ref{P7.LocalizationOfTheEnergy} and (\ref{bound_in_ring}),
\begin{equation*}
\frac{1}{2} \int_{B(0,2)}\left\{|\nabla w_\v|^2 + \frac{b^2}{2 \xi^2} (1 - |w_\v|^2)^2\right\} \leq \pi d_0 |\ln \xi| + C_0.
\end{equation*}
We may now apply Proposition 0.1 in \cite{DelFel} to $w_\v$ in $B(0,2)$ to conclude that $\displaystyle\frac{1}{\xi^2} \int_{B(0,2)} (1-|w_\v|^2)^2 \leq C_1$.
Therefore the bound (\ref{Pohozaev_bound}) holds.
\end{proof}

\subsection{The bad discs}
Consider a family of discs $(B(x_i,\v^{{1/4}}))_{i\in I}$ such that (here $I$ depends on $\v$)
\[
\text{for all }i\in I\text{ we have }x_i\in\O,
\]
\[
B(x_i,\v^{{1/4}}/4)\cap B(x_i,\v^{{1/4}}/4)=\emptyset\text{ if }i\neq j,
\]
\[
\cup_{i\in I}B(x_i,\v^{{1/4}})\supset\O.
\]
For $\mu\in(1/2,1)$, let $C=C(\mu),\,\v_0=\v_0(\mu)$ be defined as in the second part of Lemma \ref{L7.BadDiscsLemma}. For $\v<\v_0$, we say that
$B(x_i,\v^{{1/4}})$ is {\it $\mu$-good disc} if
\[
F_\v(v_\v,B({x_i,\v^{{1/4}}})\cap\O)\leq C(\mu)|\ln\v|
\]
and $B(x_i,\v^{{1/4}})$ is {\it $\mu$-bad disc} if
\begin{equation}\label{WhenThereisAMuBadDisc}
F_\v(v_\v,B({x_i,\v^{{1/4}}})\cap\O)> C(\mu)|\ln\v|.
\end{equation}
Let $J_\v=J:=\{i\in I\,|\,B({x_i,\v^{{1/4}}})\text{ is a $\mu$-bad disc}\}$.

\begin{lem}\label{L7.finitenumberbaddiscs}
There is an integer $N$, which depends only on $g$ and $\mu$, s.t.
\[
{\rm Card}\,J\leq N.
\]
\end{lem}
\begin{proof}
Since each point of $\O$ is covered by at most $16$ discs $B({x_i,\v^{{1/4}}})$, we have
\[
\sum_{i\in I}F_\v(v_\v,B({x_i,\v^{{1/4}}})\cap\O)\leq 16F_\v(v_\v,\O).
\]
The previous assertion implies that $\text{Card }J\leq\frac{16C_0}{C_1(\mu)}$.
\end{proof}
The next result is a straightforward variant of Theorem IV.1 in \cite{BBH}.
\begin{lem}\label{L7.SeparationBadDiscs}
Possibly after passing to a subsequence and relabeling $I$, we may choose $J'\subset J$ and a constant $\lambda\geq 1$ (independently of $\v$) s.t.
\[
J'=\{1,...,N'\},\,N'={\rm Cst},
\]
\[
|x_i-x_j|\geq8\lambda\v^{1/4}\text{ for }i,j\in J',\,i\neq j
\]
and
\[
\cup_{i\in J}{B(x_i,\v^{1/4})}\subset \cup_{i\in J'}{B(x_i, \lambda\v^{1/4})}.
\]
\end{lem}
We will say that, for $i\in J'$, $B(x_i,\lambda\v^{1/4})$ are {\it separated $\mu$-bad discs}. From now on, we work with separated $\mu$-bad discs.
Denote $\displaystyle\hat x_i=\frac{x_i}{\delta}$. By Proposition \ref{P7.ConvOfBadDiscsTo0} we know that for small $\v$, we have $\hat x_i\in\overline{B_1}$. Clearly, up to a subsequence,
\begin{equation}\label{7.DefintionOfTheLimitingSite}\left\{\begin{array}{c}\text{there are $\a_1,...,\a_\kappa$, $\kappa$ distinct points in $\overline{B_1}$ }
\\
\text{$\{\Lambda_1,...,\Lambda_\kappa\}$ a partition (in non empty sets) of $J'$ s.t.}
\\
\text{for }i\in J',\text{ if }i\in\Lambda_k\text{ then }\hat x_i\to \a_k.
\end{array}\right.
\end{equation}
Note that for $i\in J'$,  we have
\begin{equation}\label{7.EquivDefOfA}
y\in\{\a_1,...,\a_\kappa\}\Longleftrightarrow \begin{cases}\forall\,\eta>0\text{, for small $\v$,}\\\text{there is a $\mu$-bad disc inside $B(y,\eta)$.}\end{cases}
\end{equation}
\subsection{Convergence in $H^1_{\rm loc}(\R^2\setminus\{\a_1,...,\alpha_\kappa\})$}
We have the following theorem.
\begin{prop}\label{P7.AsymptoticResultatNonZeroDeg}
Let $\alpha_1,...,\alpha_\kappa$ be defined by \eqref{7.DefintionOfTheLimitingSite}. Then we have:
\begin{enumerate}
\item The points $\a_1,...,\a_\kappa$ belong to $\omega$.
\item There exists $v_0 \in H^1_{\rm loc}(\R^2\setminus\{\a_1,...,\a_\kappa\}, \mathbb{S}^1)$ s.t. (possibly after extraction)
\begin{equation}\label{7.H1LocConvInLittleOmega}
\text{$\hat v_{\v} \to v_0$ in  $H^1_{\rm loc}(\R^2\setminus\{\a_1,...,\a_\kappa\})$}
\end{equation}

\begin{equation}\label{7.LInftyLocConvInLittleOmega}
\text{$\hat v_{\v} \to v_0$ in  $C^0_{\rm loc}(\R^2\setminus\{\a_1,...,\a_\kappa\})$}.
\end{equation}
\item There exists $\eta_0>0$ s.t. for all $0<\eta<\eta_0$ and for sufficiently small $\v$ we have
\[
\deg_{\p B_\eta(\a_k)}(\hat v_{\v}/|\hat v_{\v}|)=\deg_{\p B_{\eta_0}(\a_k)}(v_0)=1.
\]
\item $\kappa=d_0$.
\end{enumerate}
\end{prop}
\begin{proof}{\bf Step 1:} $\hat{v}_{\v} \weak v_0$ in $H^1_{\rm loc}(\R^2\setminus\{\a_1,...,\a_\kappa\})$, $v_0\in H^1_{\rm loc}(\R^2\setminus\{\a_1,...,\a_\kappa\},\S^1)$ and $\alpha_k\in\o$

Proposition \ref{P7.ConvOfBadDiscsTo0} guarantees that $\a_1,...,\a_\kappa\in \overline{\omega}$. Let
 \begin{equation}\label{7.DefinitionOfEta0}
 \eta_0=\begin{cases}10^{-2}\cdot\min_{k \neq k'}|\a_k-\a_{k'}|&\text{if }\kappa>1\\1&\text{if }\kappa=1\end{cases}.
 \end{equation}

Applying Theorem 4.1 in \cite{SS1} we have for all $0<\eta<\eta_0$ and for small $\v$
\begin{equation}\label{7.ConcentrationOfEnergyNearBadPoints}
\frac{1}{2}\int_{\cup_{k\in \{1,...,\kappa\}} B(\a_k, \eta)}{\left\{|\n \hat v_{\v}|^2+\frac{b^2}{2\xi^2}(1-|\hat v_{\v}|^2)^2\right\}}\geq\pi d_0 \ln\frac{\eta}{\xi}-C
\end{equation}
with $C$ independent of $\v$ and $\eta$. Combining \eqref{7.ConcentrationOfEnergyNearBadPoints} with \eqref{7.EnergyOfTheHole} and Corollary \ref{C7.BoundEnergyU} we obtain that $\hat{v}_{\v}$ is bounded in $H^1(K)$; here $K\subset\R^2\setminus\{\a_1,...,\a_\kappa\}$ is an arbitrary compact set. Therefore, there exists $v_0 \in H^1_{\rm loc}(\R^2\setminus\{\a_1,...,\a_\kappa\})$ s.t. we have  $\hat{v}_{\v} \weak v_0$ in $H^1_{\rm loc}(\R^2\setminus\{\a_1,...,\a_\kappa\})$  (possibly passing to a subsequence). Since $\|1-|\hat{v}_{\v}|\|_{L^2(K)}\to0$ for all compact sets $K\subset \R^2\setminus\{\a_1,...,\a_\kappa\}$, we find that $v_0$ is $\mathbb{S}^1$- valued.

%
%
%
Following the proof of Step 7 in Theorem C in \cite{LM1}, we can prove that $\a_1, ..., \a_\kappa \notin \d \w$, thus  $\a_1, ..., \a_\kappa \in \w$, and the first assertion follows. \\

{\bf Step 2:} Proof of {\it 2.}

Adapting the techniques of \cite{BBH1} (Theorem 2, Step 1), we establish \eqref{7.H1LocConvInLittleOmega} and \eqref{7.LInftyLocConvInLittleOmega} in a ball $B=B(y,R_0)$ s.t. $\overline{B}\subset \R^2\setminus\{\a_1,...,\a_{\kappa}\}$.

Let $y\in\R^2$ and let $R'>R>0$ be s.t. $B(y,R')\subset \R^2\setminus\{\a_1,...,\a_{\kappa}\}$. Since $\hat F_{\xi}(\hat v_{\v},B(y,R'))$ is bounded independently on $\v$, there is $R_0\in(R,R')$ (independent of $\v$) s.t., passing to a further subsequence
if necessary we have
\begin{equation}\label{7.BouondOnCircleToProveStrogConvInH1}
\int_{\p B(y,R_0)}{\left\{|\p_\tau \hat v_{\v}|^2+\frac{1}{\xi^2}(1-|\hat v_{\v}|^2)^2\right\}}\leq C\text{ with $C$ independent of $\v$}.
\end{equation}
Indeed, for $r\in(R,R')$ denote
\[
I_\v(r)=\int_{\p B(y,r)}{\left\{|\n \hat v_{\v}|^2+\frac{1}{\xi^2}(1-|\hat v_{\v}|^2)^2\right\}}.
\]

Using the Fubini theorem and the Fatou Lemma we have
\[
0\leq\int_{R}^{R'}{\liminf_\v I_\v(r)\,{\rm d}r}\leq\liminf_\v\int_{R}^{R'}{I_\v(r)\,{\rm d}r}\leq C'.
\]
Consequently, $\liminf_\v I_\v(r)<\infty$ for almost all $r\in(R,R')$, so that \eqref{7.BouondOnCircleToProveStrogConvInH1} holds with $C=\displaystyle\frac{C'}{R'-R}$.

Let $g_\v=\tr_{\p B} \hat v_{\v }$. Since $|\hat{v}_{\v }|\geq 1/2$ in $B=B(y,R_0)$, we have $\deg_{\p B}(g_\v)=0$. The bound \eqref{7.BouondOnCircleToProveStrogConvInH1} implies that,
up to choose a subsequence, $g_\v$ is weakly convergent in $H^1(\p B)$. Consequently there is $h\in H^1(\p B, \S^1)$, $h={\rm e}^{\imath\f},\,\f\in H^1(\p B,\R)$ s.t.
\begin{equation}\label{7.LinftyConvData}
g_\v\to h\text{ uniformly on $\p B$},
\end{equation}
\begin{equation}\label{7.LinftyH1/2}
g_\v\to h\text{ in $H^{1/2}(\p B)$}.
\end{equation}
Let $\eta_\v:B\to\R^+$ be the minimizer of $\displaystyle\int_{B}{\left\{|\n \eta|^2+\frac{1}{\xi^2}(1-\eta)^2\right\}}$ in $H^1_{|g_\v  |}(B,\R)$. Then $\eta_\v  $ satisfies
\begin{equation*}
\begin{cases}-\xi^2\Delta\eta_{\v }+\eta_{\v }=1&\text{in }B\\\eta_{\v }=|g_\v  |&\text{on }\p B\end{cases}.
\end{equation*}
It follows from \cite{BBH1} that
\begin{equation}\label{7.Fundamentalestimateontesteta}
\int_B{\left\{|\n \eta_\v  |^2+\frac{1}{\xi^2}(1-\eta_\v )^2\right\}}\leq C\xi.
\end{equation}
Using \eqref{7.LinftyConvData}, there is $\f_\v \in H^1(\p B,\R)$, s.t. $g_\v =|g_\v |{\rm e}^{\imath\f_\v }$ and $\f_\v \to\f$ uniformly on $\p B$.
Following \cite{BBH1}, denote by $\psi_\v \in H^1_{\f_\v  }(B,\R)$ the unique solution of $-\Div(a^2\n\psi_\v )=0$. (Here $a=b$ in $\o$ and $a=1$ in $\R^2\setminus\o$.) From \eqref{7.LinftyH1/2}, $\psi_\v \to\psi$ in $H^1(B)$ where $\psi\in H^1_\f(B,\R)$ is the unique solution of $-\Div(a^2\n\psi)=0$. Since $\eta_\v {\rm e}^{\imath\psi_\v }\in H^1_{g_\v }(B)$, we have
\begin{equation}\label{7.UpperBoundOnCompactToGetH1Conv}
\hat F_{\xi}(\hat v_{\v },B)\leq\hat  F_{\xi}(\eta_\v {\rm e}^{\imath\psi_\v },B) \leq\frac{1}{2}\int_B{\hat U_{\v}^2|\n \psi_\v |^2}+C\xi\underset{\v\to0}{\to}\frac{1}{2}\int_B{a^2|\n \psi|^2}.
\end{equation}
On the other hand, since $\hat v_{\v}\weak v_0$ in $H^1(B)$, we have $v_0={\rm e}^{\imath\phi}$ with $\phi\in H^1_{\f}(B,\R)$ and
\begin{equation}\label{7.ToGetLowerBoundH1Conv}
\liminf_\v \hat F_{\xi}(\hat v_{\v},B)\geq\liminf_\v \frac{1}{2}\int_B{\hat U_{\v}^2|\n \hat v_{\v}|^2}\geq\frac{1}{2}\int_{B}{a^2|\n v_0|^2}=\frac{1}{2}\int_{B}{a^2|\n \phi|^2}.
\end{equation}
(The last inequality follows from $\hat{U}_{\v}\to a$ in $L^2$, $|U_{\v}|\leq1$ and $\hat{v}_{\v}\weak v_0$ in $H^1$.)

By combining \eqref{7.UpperBoundOnCompactToGetH1Conv}, \eqref{7.ToGetLowerBoundH1Conv}  and the fact that $\psi$ minimizes $\displaystyle\int_B{a^2|\n \cdot|^2}$ in $H^1_\f(B,\R)$,  we find that (\ref{7.H1LocConvInLittleOmega}) holds. Furthermore, the map $\psi$ in \eqref{7.UpperBoundOnCompactToGetH1Conv} is the same as $\phi$ in \eqref{7.ToGetLowerBoundH1Conv}.

Note that since
\[
\displaystyle\frac{1}{2}\int_B\hat{U}_{\v}^2\left|\n \frac{\hat{v}_{\v}}{|\hat{v}_{\v}|}\right|^2-o_\v(1)\leq\hat F_{\xi}(\hat v_{\v},B),
\]
by comparing \eqref{7.UpperBoundOnCompactToGetH1Conv} with \eqref{7.ToGetLowerBoundH1Conv}, we also have
\begin{equation}\label{7.potentialBoundincompact}
\int_{K}{|\n|\hat{v}_{\v}||^2+\frac{1}{\xi^2}(1-|\hat v_{\v}|^2)^2}\to0.
\end{equation}
In order to prove \eqref{7.LInftyLocConvInLittleOmega},  it suffices to establish the convergence
\begin{equation}\label{7.LInftyLocConvInLittleOmegaPhase}
\phi_\v \to\phi\text{ in }L^\infty(B)\text{ with }\phi_\v \in H^1_{\f_\v }(B,\R)\text{ and }\hat v_{\v }=|\hat v_{\v }|{\rm e}^{\imath\phi_\v },
\end{equation}
and to use the fact that $|\hat{v}_{\v }|\to1$ uniformly.

Proof of \eqref{7.LInftyLocConvInLittleOmegaPhase}. If $\p\o\cap B=\emptyset$, then the argument is the same as in \cite{BBH1}. Assume next that $\p\o\cap B\neq\emptyset$, and let $\tilde\psi\in H^{3/2}(B,\R)$ be the harmonic extension of $\f$. Since $\zeta:= \phi-\tilde\psi\in H^1_0(B,\R)$ satisfies $-\Div(a^2\n\zeta)=\Div(a^2\n\tilde\psi)$, Theorem 1 in \cite{NGMeyers1} implies that $\phi\in W^{1,p}(B,\R)$ for some $p>2$.

We next prove that, for some $q>2$ and $\tilde{B}=B(y',\tilde{R})$ s.t. $B(y',2\tilde{R})\subset B$, we have $\|\phi_\v -\phi\|_{W^{1,q}(\tilde{B})}\to 0$. (Once proved, this assertion will imply, via Sobolev embedding  that \eqref{7.LInftyLocConvInLittleOmegaPhase} holds.)

Note that (up to a subsequence) $\phi_\v \to\phi$ in $L^2(B,\R)$. Thus we have
\[
\begin{cases}-\Div\left[\hat U_\v^2|\hat v_{\v }|^2\n(\phi_\v -\phi)\right]=\Div\left[(\hat U_{\v }^2|\hat v_{\v }|^2-a^2)\n\phi\right]&\text{in }B
\\
\|\phi_\v -\phi\|_{L^2(B)}\to0\end{cases}.
\]
From Theorem 2 in \cite{NGMeyers1}, there is $2<q\leq p$ and $C>0$ s.t.
\[
\|\n(\phi_\v -\phi)\|_{L^q(\tilde{B})}\leq C\left(\tilde{R}^{-2+2/q}\|\phi_\v -\phi\|_{L^2(B)}+\|(\hat U_{\v }^2|\hat v_{\v }|^2-a^2)\n\phi\|_{L^q(B)}\right)\underset{\v\to0}{\to}0.
\]
Consequently, $\|\phi_\v -\phi\|_{W^{1,q}(\tilde{B})}\to 0$.


%
%

{\bf Step 3:} We prove the third assertion

Let $\eta_0>\eta>0$, with $\eta_0$ defined by \eqref{7.DefinitionOfEta0}. Denote $d_{k}=\deg_{\p B(\a_k, r)}(v_0)$. These integers  do not depend on $r \in(\eta,\eta_0)$. Moreover, we have $\sum_k d_k=d_0$. For $r\in(\eta,\eta_0)$, we obtain that
\[
2\pi |d_k|\leq\int_{\p B(\a_k, r)}|\p_\tau v_0|\leq\sqrt{2\pi r}\left(\int_{\p B(\a_k, r)}\left|\p_\tau v_0\right|^2\right)^{1/2},
\] and therefore
\begin{equation}\nonumber
\frac{1}{2}\int_{B(\a_k, \eta_0)\setminus \overline{B(\a_k,\eta)}}{\left|\n v_0\right|^2}\geq\pi d_k^2\ln\frac{\eta_0}{\eta}.
\end{equation}
Consequently, we have
\begin{eqnarray}\nonumber
\liminf \frac{1}{2}\int_{\cup_k B(\a_k, \eta_0)\setminus \overline{B(\a_k,\eta)}}{\left|\n \hat v_{\v }\right|^2}&\geq&\frac{1}{2}\int_{\cup_k B(\a_k, \eta_0)\setminus \overline{B(\a_k,\eta)}}{\left|\n v_0\right|^2}
\\\label{7.EstimationDegWithSquare}
&\geq&\pi \sum_k d_k^2\ln\frac{\eta_0}{\eta}.
\end{eqnarray}
By combining \eqref{7.ConcentrationOfEnergyNearBadPoints} and \eqref{7.EstimationDegWithSquare}, we obtain the existence of $C$ independent of $\v$ and $\eta$ s.t. 
\begin{eqnarray*}
\frac{1}{2}\int_{\cup_k B(\a_k,\eta_0)}{\left|\n \hat v_{\v }\right|^2}&\geq& \pi \sum_k d_k^2\ln\frac{\eta_0}{\eta}+\pi d_0 \ln\frac{\eta}{\xi}-C\\&=&\pi d_0 \ln\frac{\eta_0}{\xi}+\pi (\sum_k d_k^2-d_0)\ln\frac{\eta_0}{\eta}-C.
\end{eqnarray*}
Therefore, $d_k$ must be either $0$ or $1$. Otherwise, \eqref{7.EnergyOfTheHole} cannot hold for small $\eta$. Applying the strong convergence result from Step 2 with $K=B(\a_k, \eta) \setminus \overline{B(\a_k, \frac{\eta}{2})}$, we have that for
small $\v$,
\[
d_k=\deg_{\p B(\a_k, \eta)}\left(\frac{\hat v_{\v }}{|\hat v_{\v }|}\right).
\]
We next prove that $d_k=1$ for each $k$. By contradiction, assume that there is $k_0$ s.t. $d_{k_0}=0$. We may assume that $k_0=1$. From \eqref{7.EquivDefOfA}, there is a (separated) $\mu$-bad disc $B(\hat{x}_0,\lambda\v ^{{1/4}}/\delta)$ in $B(\alpha_1,\eta_0)$. Thus by  \eqref{WhenThereisAMuBadDisc}, we have
$$\hat{F}_{\xi}(\hat{v}_{\v },B(\hat{x}_0,\lambda\v ^{{1/4}}/\delta))>C(\mu)|\ln\v|.$$
On the other hand, since in $|\hat{v}_{\v }|\geq1/2$ in $B(\alpha_k,\eta_0)\setminus\overline{B(\alpha_k,\eta_0/2)}$, applying Theorem 4.1 in \cite{SS1} in $B(\alpha_k,\eta_0)$, $k\in\{2,...,\kappa\}$, with $r=10^{-4}\cdot\eta_0$ we find that
\[
\hat{F}_\xi(\hat{v}_{\v },B(\alpha_k,\eta_0))\geq b^2|\deg_{\p B(\alpha_k,\eta_0)}(\hat{v}_{\v })||\ln\xi|-C, \ k = 2,...,\kappa.
\]
Since $\sum_{k=2}^\kappa\deg_{\p B(\alpha_k,\eta_0)}(\hat{v}_{\v })=d_0$, the above estimates yield
\[
\hat{F}_\xi(\hat{v}_{\v },B(0,\frac{\rho}{\delta}))\geq b^2d_0|\ln\xi|+C(\mu)|\ln\v|-C
\]
which is in contradiction with \eqref{7.MainHyp} and \eqref{7.UpperboundAuxPb}. Thus $d_k=1$ for $k\in\{1,...,\kappa\}$ and consequently, $\kappa=d_0$.
\end{proof}


We are now in position to estimate the rate of uniform convergence of $|\hat{v}_{\v }|$ in a compact set $K\subset\R^2\setminus\{\alpha_1,...,\alpha_{d_0}\}$.
\begin{cor}\label{C7.UnifCongLocNonZeroDeg}
There is $C>0$ s.t. for $\eta_0>\eta > 0$ and small $\v$ we have
\begin{equation*}
|\hat v_{\v }|\geq1- C |\ln\v |^{-1/3}\text{ in }B(0,\frac{\rho}{\delta})\setminus \overline{B(\a_i, \eta)}.
\end{equation*}
\end{cor}
\begin{proof}
Due to \eqref{5.EstMofUsiMirArg}, it is sufficient to establish this result in $B(0,1)\setminus \overline{B(\a_i, \eta)}$. Combining Corollary \ref{C7.Pohozaev_bound} with \eqref{7.H1LocConvInLittleOmega}, we obtain that
\[
\hat{F}_\xi(\hat{v}_{\v },B(0,2)\setminus \overline{B(\a_i, \eta/2)})\leq C(\eta).
\]
Thus for all $x\in B(0,\rho)$ s.t. $B(\hat{x},\v ^{1/4}/\delta)\subset B(0,2)\setminus \overline{B(\a_i, \eta/2)}$, for small $\v$ we have
\[
F_\v(v_{\v },B(x,\v ^{1/4}))\leq\hat{F}_\xi(\hat{v}_{\v },B(0,2)\setminus \overline{B(\a_i, \eta/2)})<|\ln\v |^{1/3}.
\]
From Lemma \ref{L7.BadDiscsLemma} (first assertion), we obtain the existence of $C>0$ (independent of $\v$ and $\eta$) s.t. $|v_{\v }(x)|=|\hat{v}_{\v }(\hat{x})|\geq 1- C |\ln\v |^{-1/3}$. Finally, since for all $\hat{x}\in B(0,1)\setminus \overline{B(\a_i, \eta)}$ we have $B(\hat{x},\v ^{1/4}/\delta)\subset B(0,2)\setminus \overline{B(\a_i, \eta/2)}$, Corollary \ref{C7.UnifCongLocNonZeroDeg} follows.
\end{proof}

\subsection{Information about the limit $v_0$ }
Following \cite{BBH} (Appendix IV, page 152) we have
\begin{prop}\label{P7.LpBoundNonZeroDeg}
For all $1\leq p<2$ and for any compact $K \subset \R^2$, $\hat v_{\v }$ is bounded in $W^{1,p}(K)$.
\end{prop}

Let $\theta_i$ be the main argument of $\frac{\hat{x}-\alpha_i}{|\hat{x}-\alpha_i|}$ and set $\theta=\theta_1+...+\theta_{d_0}$. Note that $\n\theta$ is smooth away from $\{\a_1,...,\a_\kappa\}$ and $\Pi_i\frac{\hat{x}-\a_i}{|\hat{x}-\a_i|}={\rm e}^{\imath\theta}$. Let $\tilde g := \tr_{\d B_1} v_0$ and $\f_0\in C^\infty(\p B_1,\R)$ be s.t.
$\tilde{g}=\Pi_i\frac{\hat{x}-\a_i}{|\hat{x}-\a_i|}{\rm e}^{\imath\f_0}={\rm e}^{\imath(\theta+\f_0)}$ (see \cite{glcoursep6} for the existence of $\f_0$).
\begin{prop}\label{P7.equationforphistar} The limit $v_0$ satisfies $-\Div\left(a^2\,v_0\times\n v_0\right)=0$ in $\mathcal{D}'(\R^2)$.
Moreover we may write $v_0={\rm e}^{\imath(\theta+\f_\star)}$. Here $\f_\star$ is the solution of
\begin{equation}\label{7.EquationForPhiStar}
\begin{cases}-\Div\left[a^2\n(\theta+\f_\star)\right]=0&\text{in }B_1
\\
\f_\star=\f_0&\text{on }\p B_1
\end{cases}.
\end{equation}
\end{prop}

\begin{proof}
Let $\phi\in\mathcal{D}(\R^2)$, and set $K = {\rm supp}(\phi)$. By Proposition \ref{P7.LpBoundNonZeroDeg}, we have $\hat U_{\v }^2\,\hat v_{\v }\times\n\hat v_{\v }\weak a^2\,v_0\times\n v_0$ in $L^p(K)$ for $p<2$. Multiplying the equation
$-\Div\left[\hat U_{\v }^2\,\hat v_{\v }\times\n\hat v_{\v }\right]=0$ by $\phi$ and integrating by parts, we obtain
\begin{eqnarray*}
0=\int_{K}{-\Div\left[\hat U_{\v }^2\,\hat v_{\v }\times\n\hat v_{\v }\right]\phi}&=&\int_{K}{\hat U_\v^2\,\hat v_{\v }\times\n\hat v_{\v }\cdot\n\phi}\\&\to&
\int_{K}{a^2\,v_0\times\n v_0\cdot\n\phi}=\int_{K}{-\Div\left(a^2\,v_0\times\n v_0\right)\phi}.
\end{eqnarray*}
Consequently $-\Div\left(a^2\,v_0\times\n v_0\right)=0$ in $\mathcal{D}'(\R^2)$.


In order to prove that $-\Div\left[a^2\n(\theta+\f_\star)\right]=0$ in $B_1\setminus \{\a_{1},...,\alpha_{d_0}\}$, we follow \cite{LM1}, Step 12 of Theorem C.

Next, we prove that $\f_{\star}$ is harmonic in a neighborhood of $\a_{k}$.
Fix $\l>0$ and $x_0 \in \w$ s.t. $B(x_0, 2\l) \subset \w \setminus \{\a_1,...,\alpha_{d_0}\}$. As we established in Proposition \ref{P7.AsymptoticResultatNonZeroDeg}, Step 2,
$F_{\xi}(\hat v_{\v }, B(x_0, 2 \l))$ is uniformly bounded in $\v$. Proceeding as in Step 2, we conclude that exists $\l_0 \in (\l, 2\l)$ s.t., after
passing to a further subsequence, we have
\begin{equation}\label{boundaryest}
\int_{\d B(x_0, \l_0)}\left\{|\nabla \hat v_{\v }|^2 + \frac{1}{2 \xi^2}(1-|\hat v_{\v }|^2)^2\right\} \leq C
\end{equation}
with $C$ and $\l_0$ independent of $\v$. Now, if $\hat u_\v $ minimizes
\[
\hat E_{\xi}(\hat u) = \frac{1}{2}\int_{B(0,\frac{\rho}{\delta})}\left\{|\nabla \hat u|^2 +\frac{1}{2 \xi^2}  (a^2 - |\hat u|^2)^2\right\}
\]
subject to $\hat u(x) = f_{\v }(\delta x)$ on $\d B(0,\frac{\rho}{\delta})$, then $\hat{u}_\v $ minimizes $\hat E_{\xi}(\hat u, B(x_0, \l_0))$ with
respect to its own boundary conditions. In other words, $\hat w_\v : = \frac{\hat u_{\v }}{b}$ minimizes the classical energy
\[
\frac{1}{2}\int_{B(x_0, \l_0)}\left\{|\nabla \hat w_\v |^2 + \frac{b^2}{2 \xi^2} (1 - |\hat w_\v |^2)^2\right\}
\]
among $w \in H^1(B(x_0, \l_0))$ such that $w = h_\v :=  \frac{\hat{u}_{\v }}{b}$ on $\d B(x_0, \l_0)$. It follows from (\ref{boundaryest})
and Proposition \ref{P7.UepsCloseToaeps} that $h_\v $ also satisfies
\begin{equation}\label{boundaryest1}
\int_{\d B(x_0, \l_0)}\left\{|\p_\tau h_\v |^2 + \frac{1}{2 \xi^2}(1-|h_\v |^2)^2\right\} \leq C + 1.
\end{equation}
Note that by Proposition \ref{P7.UepsCloseToaeps} we have
\begin{equation}\label{max_princ}
\|\hat w_\v \|_{L^{\infty}(B(x_0, \lambda_0))} \leq 1 + c e^{-c_0 \xi}.
\end{equation}
Using (\ref{max_princ}) and the uniform bound from Corollary \ref{C7.UnifCongLocNonZeroDeg}, we may repeat the arguments of Theorem 2 in \cite{BBH1} and conclude that, up to a subsequence, there exists an $\S^1$-valued map $w_0$ s.t. for every compact
$K \subset (\w \setminus \{\a_1,...\alpha_{d_0}\})$ we have
\begin{equation}\label{C^k conv}
\hat w_\v  \to w_0 \text{ in } C^{\infty}(K)
\end{equation}
and
\begin{equation}\label{C^k conv_grad}
\frac{b^2 (1 - |\hat w_\v |^2)}{\xi^2} \to |\nabla w_0|^2 \text{ in } C^{\infty}(K).
\end{equation}
Fix now $r <\min\left\{\dfrac{\min|\a_k - \a_j|}{8}, \dfrac{\dist(\a_k, \d \w)}{8}\right\}$ and denote
$\w_r:= \{x \in \w, {\rm dist}(x, \d \w) > r\}$. It follows from (\ref{C^k conv}) that $\hat w_\v  \to q_0 := {\rm tr}_{\p \w_r} w_0$ in
$C^{\infty}(\p \w_r)$. In view of Proposition \ref{P7.LpBoundNonZeroDeg}, we have $w_0 \in W^{1,p}(\w_r)$, $p<2$. By Remark I.1 in \cite{BBH}, this implies that
\[
w_0 = \tilde w \, {\rm exp}\left(i \sum_{k} c_k \ln |x-\a_k| + i \chi \right).
\]
Here:
\begin{enumerate}[$\bullet$]
\item $\tilde w$ is the {\it canonical harmonic map} (see \cite{BBH}, Sec. I.3.) having singularities $\{\a_k, k = 1, ..., d_0\}$
and equal to $q_0$ on $\p \w_r$;
\item the $c_k$'s are real coefficients;
\item $\chi$ is the solution of
\[
\begin{cases}
\Delta \chi = 0  &\text{in} \ \w_r\\
\chi(x) + \sum_{k} c_{k} \ln|x - \a_{k}| = 0  &\text{on} \ \p \w_r
\end{cases}.
\]
\end{enumerate}
Repeating the argument of \cite{BBH}, Theorem VII.1, Step 2 (the key ingredients of this proof are (\ref{C^k conv}), (\ref{C^k conv_grad}) and
Corollary \ref{C7.Pohozaev_bound}),
we find that $c_k \equiv 0, k = 1, ..., d_0$, and, consequently, $w_0 \equiv \tilde w$ in $\w_r$. Finally, by \cite{BBH}, Corollary I.2.,
we know that the canonical harmonic map $\tilde w$ is of the form $\tilde w = e^{i(\theta + \f_{\star})}$ with $\f_{\star}$ harmonic in $\w_r$.

\end{proof}

\subsection{Uniqueness of zeros}
\begin{prop}\label{P7.uniq0}
For $\v$ sufficiently small, the minimizer $\hat v_\v$ has exactly $d_0$ zeros.
\end{prop}
\begin{proof}

It suffices to prove that for small $\v$ there is a unique zero of $\hat v_{\v}$ in $B(\a_k, r)$, $k = 1,...,d_0$, with $r$ defined in the proof of Proposition \ref{P7.equationforphistar}.

Since $\hat w_\v  = \frac{\hat v_{\v } \hat U_{\v }}{b}$, from Proposition \ref{P7.UepsCloseToaeps} and Proposition \ref{P7.equationforphistar} we see that $w_0 = v_0 = e^{i(\theta_k + H_k)}$ in $B(\a_k, r)$, where $\theta_k$ is the phase of $\frac{x-\a_k}{|x-\a_k|}$ and $H_k = \f_{\star} + \psi_k$ is harmonic in $B(\a_k, r)$. Using (\ref{C^k conv}) and (\ref{C^k conv_grad}) and arguing as  in the alternative proof of  Theorem VII.4 in \cite{BBH} (page 74) we obtain that $\nabla H_k(\a_k) = 0$.

Finally, we are now in position to obtain, as in Theorem IX.1 \cite{BBH} (using the main result of \cite{BCP1}), that there is a unique zero of $\hat w_\v $ (and, therefore, of $\hat v_{\v }$) in $B(\a_k, r)$.

\end{proof}

\subsection{Summary}
We have thus proved
\begin{thm}\label{T7.THM3}
Let $\v_n\downarrow0$ and $\hat v_{\v_n}$ be a minimizer of \eqref{variable_change} in \eqref{modelclassblownup} for $\v=\v_n$. Then there exist $d_0$ distinct points $\a_1,...,\a_{d_0} \in\o$ and a function $v_0 \in H^1_{\rm loc}(\R^2\setminus\{\a_1,...,\a_{d_0}\},\S^1)\cap W^{1,p}_{\rm loc}(\R^2,\S^1)$ ($p<2$) s.t., up to a subsequence
\begin{enumerate}
\item $\hat v_{\v_n} \to v_0$ in $H^1_{\rm loc}(\R^2\setminus\{\a_1,...,\a_{d_0}\})$ and $C^0_{\rm loc}(\R^2\setminus\{\a_1,...,\a_{d_0}\})$,
\item $\hat v_{\v_n} \weak v_0$ in $W^{1,p}_{\rm loc}(\R^2)$ ($p<2$),
\item for $K\Subset\R^2\setminus\{\a_1,...,\a_{d_0}\}$, $|\hat v_{\v_n}|\geq1-|\ln\v_n|^{-1/3}\text{ in }K$ and $\displaystyle\int_{K}{\left|\n|\hat{v}_{\v_n}|\right|^2+\frac{1}{\xi^2}(1-|\hat v_{\v_n}|^2)^2}\to0$,
\item for $K\Subset\R^2\setminus\overline{\o}$, $\hat v_{\v_n}\to v_0$ in $C^\infty(K)$ and $1-|\hat v_{\v_n}|\leq C_K\xi^2$,
\item $\hat v_{\v_n}$ has exactly $d_0$ zeros $x_1^n, ..., x_{d_0}^n$ and $x_i^n \to \a_i$,
\item $v_0$ satisfies $-\Div\left(a^2\,v_0\times\n v_0\right)=0$ in $\mathcal{D}'(\R^2)$.
\end{enumerate}
\end{thm}

Let us summarize the proof of Theorem \ref{T7.THM3}:
\begin{enumerate}[$\bullet$]
\item Statement {\it 1.} is established in Proposition \ref{P7.AsymptoticResultatNonZeroDeg},
\item Statement {\it 2.} follows from Propositions \ref{P7.LpBoundNonZeroDeg} and \ref{P7.equationforphistar},
\item  Statement {\it 3.} is a consequence of Corollary  \ref{C7.UnifCongLocNonZeroDeg} and (\ref{7.potentialBoundincompact}),
\item  Statement {\it 4.} is Corollary \ref{C7.SmoothConv},
\item  Statement {\it 5.} is proved in Proposition \ref{P7.uniq0},
\item  Statement {\it 6.} is established in Proposition \ref{P7.equationforphistar}.
\end{enumerate}
The proof of Theorem \ref{T7.THM3} is complete.


\section{Renormalized energy for the model problem}

In this section, we establish the expansion for $F_\v(v_\v,B(0,\rho))=\hat F_{\xi}(\hat v_\v,B(0, \frac{\rho}{\dl}))$; specifically, we derive the expression for
\begin{equation}\label{7.FULLRENENALTERNATIVEDEFINITION}
\lim_\v\left\{\hat F_{\xi}(\hat v_\v,B(0, \frac{\rho}{\dl}))-\pi d_0^2 |\ln \delta| -\pi d_0 b^2|\ln{\xi}|\right\}.
\end{equation}
(We are going to prove that this limit exists).

In order to find an expression for \eqref{7.FULLRENENALTERNATIVEDEFINITION}, our strategy is the following:
\begin{enumerate}[$\bullet$]
\item (Section \ref{S7.AwayTheInclusEst}) We first study the minimization of the Dirichlet functional among $\S^1$-valued maps in
annulars $B(0,\rho/\delta)\setminus\overline{B(0,1)}$ with the Dirichlet boundary conditions: $f^\delta(\delta\cdot)$ on
$\p B(0,\rho/\delta)$ and $g^\delta$ on $\p B(0,1)$. Here $f^\delta=\frac{f_{\v}}{|f_{\v}|}$ where $f_{\v}$ is given in the model problem and
$g^\delta,g^0\in C^\infty(\p B_1,\S^1)$ are s.t. $g^\delta\to g^0$ in $C^1$. We get that the Dirichlet energy has the form $\pi d_0^2\ln(\rho/\delta)+\tilde W_0(f_0)+\tilde W_1(g^0)+o_\delta(1)$.
\item (Sections \ref{S7.ArroudTheInclusEst}, \ref{S7.UpBounAuxPBBBB} and \ref{S7.LowBounAuxPBBBB}) In $B(0,1)$, we study the weighted Ginzburg-Landau functional with the Dirichlet boundary condition $g^\delta$ on $\p B(0,1)$. Making use of the previous bullet point, one may obtain the matching upper and lower bounds and use them to derive the third term of renormalized energy, which depends on the limiting locations of the zeros ${\bm\beta}=(\beta_1,...,\beta_{d_0})\in\o^{d_0}$ and on $g^0$. We establish that
\[
\inf_{ H^1_{g^\delta}}\hat F_\xi(\cdot,B(0,1))=\pi d_0 b^2\ln\frac{b}{\xi}+d_0 b^2 \gamma + \tilde{W}_2({\bm \beta},g^0)+o_\v(1).
\]
\item (Section \ref{S7?MinAuxEnsdkfjh}) Finally, we make a fundamental observation: the limiting function $g_0=\lim \tr_{\p B_1}\hat v_\v$ and the points ${\bm \alpha}$ obtained form Theorem \ref{T7.THM3} form a minimal configuration for $W_1(\tilde{g})+W_2({\bm\beta},\tilde{g})$.
Thus, introducing
\[
\tilde{W}({\bm\beta})=\inf_{\substack{\tilde{g}\in C^\infty(\p B_1,\S^1)\\\text{with }\deg_{\p B_1}(\tilde{g})=d_0}}\left\{\tilde{W}_1(\tilde{g})+\tilde{W}_2({\bm\beta},\tilde{g})\right\}
\]
we conclude that ${\bm\alpha}$ minimizes $\tilde{W}$.
\end{enumerate}

In this section we prove the following theorem.
\begin{thm}\label{T7.THM4}
The following energy expansion holds when $\v\to0$
\begin{equation}\label{Ren_en_Th4}
\hat F_{\xi}(\hat v_\v,B(0, \frac{\rho}{\dl}))=\pi d_0 b^2\ln\frac{b}{\xi} +\pi d_0^2 \ln \frac{\rho}{\delta} +\tilde{W}_0(f_0) + \tilde W ({\bm \a})+d_0 b^2 \gamma + o_\v(1).
\end{equation}
Here the points ${\bm \a}=(\a_1,..., \a_{d_0})$ are obtained from Theorem \ref{T7.THM3}, $\g>0$ is an absolute constant and $\tilde{W}_0(f_0), \tilde{W}({\bm\a})$ are renormalized energies:
\begin{enumerate}[$\bullet$]
\item $\tilde{W}_0$ is independent of the points $\alpha_1,...,\alpha_{d_0}$ and given by (\ref{decoupled_energy}),
\item $\tilde{W}$ is given by \eqref{Ren_en_loc}, it is independent of $f_0$ and the limiting points $(\alpha_1,..,\alpha_{d_0})$ minimize $\tilde{W}$.
\end{enumerate}


\end{thm}
\begin{remark}
The renormalized energy in the expansion (\ref{Ren_en_Th4}) decouples into the part that depends only on the external boundary conditions
$\tilde{W_0}(f_0)$ and the part that depends only on the location of the vortices $\tilde{W}({\bm \a})$. Since ${\bm \a}$ minimizes $\tilde{W}$, the external boundary data has no effect on the location of vortices inside the inclusion. This is a drastic difference with the results of \cite{BBH} and \cite{LM1}, where the Dirichlet boundary data on the external boundary influences the location of the vortices.
\end{remark}
\subsection{Minimization among $\S^1$-valued maps away from the inclusion}\label{S7.AwayTheInclusEst}
Denote $B_\rho := B(0,\rho)$. Let
\[\text{$(f^\delta)_{0<\delta<1}\subset C^\infty(\p B_\rho, \S^1)$, $f^0 \in C^\infty(\p B_\rho, \S^1)$ be s.t. }\begin{cases}\text{$f^\delta\underset{\delta\to0}{\to}  f^0$ in $C^1(\p B_\rho)$}\\ \deg_{\p B_\rho}(f^\delta)=d_0\end{cases},
\]
and
\[
 \text{$(g^\delta)_{0<\delta<1}\subset C^\infty(\p B_1,\S^1)$, $g^0 \in C^\infty(\p B_1, \S^1)$ be s.t. }\begin{cases}g^\delta\underset{\delta\to0}{\to}  g^0\text{ in }C^1(\p B_1)\\ \deg_{\p B_1}(g^\delta)=d_0\end{cases}.
 \]
For $\delta\in(0,1)$, we denote $A_\delta=B_{\rho/\delta}\setminus \overline{B_1}$ and
\[
W_\delta=\{u\in H^{1}(A_\delta,\S^1)\,|\,\tr_{\p B_{\rho/\delta}}u(\cdot)=f^\delta (\delta\cdot)\text{ and }\tr_{\p B_1}u= g^\delta\}, 
\]
\[
Y_\delta=\{u\in H^{1}(A_\delta,\S^1)\,|\,\tr_{\p B_{\rho/\delta}}u(\cdot)=f^0(\delta\cdot)\text{ and }\tr_{\p B_1}u= g^0\}.
\]
Consider the following minimization problems:
\begin{equation}\tag{$P_\delta$}\label{7.AuxProblem}
I_\delta(f^\delta,g^\delta)=I_\delta=\inf_{u\in W_\delta}\frac{1}{2}\int_{A_\delta}{|\n u|^2}.
\end{equation}
\begin{equation}\tag{$Q_\delta$}\label{7.AuxProblemFixData}
J_\delta(f^0,g^0)=J_\delta=\inf_{u\in Y_\delta}\frac{1}{2}\int_{A_\delta}{|\n u|^2}.
\end{equation}
\begin{prop}
For small $\v$, $I_{\dl}$ is close to $J_{\dl}$, namely
\begin{equation}\label{equivIandJ}
I_{\dl} = J_{\dl} + o_{\delta}(1).
\end{equation}
\end{prop}
\begin{proof}
In this subsection $\theta$ stands for the main argument of $z$ \emph{i.e.} $\frac{z}{|z|}={\rm e}^{\imath\theta}$. For $\delta \geq 0$, let
 $\phi_\delta\in C^{\infty}(\p B_1,\R)$ be s.t. ${g}^\delta={\rm e}^{\imath (d_0\theta+\phi_\delta)}$ and $\zeta_\dl \in C^{\infty}(\p B_\rho,\R)$ be s.t.
 $f^\dl={\rm e}^{\imath(d_0\theta+\zeta_\dl)}$.
We may assume that $\phi_\delta\to\phi_0$ in $C^1(\p B_1)$ and $\zeta_\delta\to\zeta_0$ in $C^1(\p B_\rho)$. Note that
\begin{equation}\label{7.DecompositionEnergyAuxProb}
u\in W_\delta\Longleftrightarrow \,u={\rm e}^{\imath(\f+d_0\theta)}\text{ with }\f\in w_\delta.
\end{equation}
Here $w_\delta:=\{\f\in H^1(A_\delta,\R)\,|\,\tr_{\p B_{\frac{\rho}{\dl}}}\f(\cdot)=\zeta_\dl(\delta\cdot)\text{ and }\tr_{\p B_1}\f=\phi_\delta \}$.

Since $\Delta \theta = 0$ in $A_\delta$ and $\p_\nu \theta = 0$ on $\d A_\delta$, for $u\in W_\delta$ we have

\[
\int_{A_\delta}{|\n u|^2}=\int_{A_{\delta}}{|\n (\f+ d_0 \theta)|^2}=d_0^2\int_{A_\delta}{|\n \theta|^2}+\int_{A_\delta}{|\n \f|^2}.
\]
Consequently, the problem \eqref{7.AuxProblem} has a unique solution $u_\delta={\rm e}^{\imath(d_0 \theta+\varphi_\delta)}$, with $\f_\delta$ being the unique solution of
\[
\begin{cases}-\Delta\f_\delta=0&\text{in }A_\delta\\ \f_\delta(\cdot)=\zeta_\dl\left({\delta}\cdot\right)&\text{on }\p B_{\frac{\rho}{\delta}}
\\
\f_\delta=\phi_\delta&\text{on }\p B_{1}\end{cases}.
\]
With the same argument, the problem \eqref{7.AuxProblemFixData} admits a unique solution $v_\delta={\rm e}^{\imath(d_0\theta+\psi_\delta)}$ with $\psi_\delta$ being the unique solution of
\[
\begin{cases}-\Delta\psi_\delta=0&\text{in }A_\delta\\ \psi_\delta(\cdot)=\zeta_0\left(\delta\cdot\right)&\text{on }\p B_{\frac{\rho}{\delta}}
\\\psi_\delta=\phi_0&\text{on }\p B_1\end{cases}.
\]
Denote $\eta_\delta=\varphi_\delta-\psi_\delta$. Then $\eta_\delta$ is the unique solution of
\[
\begin{cases}
\Delta \eta_\delta=0&\text{in }A_\delta\\\eta_\delta=\hat \zeta_\dl - \hat \zeta_0 &\text{on }\p B_{\frac{\rho}{\dl}}\\\eta_\delta=\phi_\delta-\phi_0&\text{on }\p B_1
\end{cases}.
\]
(Here $\hat \zeta (x) := \zeta (\dl x)$). 

One may prove that $\|\psi_\delta\|_{L^2(A_\delta)}$ is bounded and more precisely we have the following result.
\begin{prop}\label{P7.decoupling}
\begin{equation}\label{7.TheBadPartGoesToZero}
\frac{1}{2}\int_{A_\delta}{|\n\psi_\delta|^2}\to|\phi_0|_{H^{1/2}(\S^1)}^2+|\zeta_0|_{H^{1/2}(\p B_\rho)}^2, \text{ as } \dl \to 0.
\end{equation}
\end{prop}

\begin{proof}
Let $(a_n)_{n\in\Z},(b_n)_{n\in\Z}\subset\C$ be s.t.
\[
\phi_0({\rm e}^{\imath\theta})=\sum_{n\in\Z}{a_n}{\rm e}^{\imath n\theta}\text{ and }\zeta_0(\rho {\rm e}^{\imath\theta})=\sum_{n\in\Z}{b_n}{\rm e}^{\imath n\theta}.
\]
We have
\begin{equation}\nonumber
|\phi_0|_{H^{1/2}(\S^1)}^2=\sum_\Z |n||a_n|^2\text{ and }|\zeta_0|^2_{H^{1/2}(\p B_{\rho})} = |\hat \zeta_0|^2_{H^{1/2}(\p B_{\frac{\rho}{\dl}})}=\sum_\Z |n||b_n|^2.
\end{equation}
From \cite{BeMi1} (Appendix D.), denoting $R(\dl) = \frac{\rho}{\dl}$, we have
\begin{eqnarray}\nonumber
\frac{1}{2\pi}\int_{A_\delta}{|\n\psi_\delta|^2}&=&\frac{|b_0-a_0|^2}{\ln R(\delta)}+\sum_{n\neq0}\frac{|n|}{R(\delta)^{2|n|}-1}\Big[(|a_n|^2+|b_n|^2)(R(\delta)^{2|n|}+1)
\\\nonumber&&\phantom{aaaaaaaaqqqqqqqqqqaaaaaassqqaaaa}-2(\overline{a_n}b_n + a_n\overline{b_n})R(\delta)^{|n|}\Big]
\\\nonumber
&=&|\phi_0|_{H^{1/2}(\d B_1)}^2+|\psi|_{H^{1/2}(\p B_\rho)}^2+\frac{|b_0 -a_0|^2}{\ln R(\delta)}
\\\nonumber
&&\phantom{aazz}+\sum_{n\neq0}\frac{2}{R(\delta)^{2|n|}-1}\left[(|a_n|^2+|b_n|^2)-(\overline{a_n}b_n +a_n\overline{b_n})R(\delta)^{|n|}\right]
\\\nonumber
&=&|\phi_0|_{H^{1/2}(\d B_1)}^2+|\zeta_0|_{H^{1/2}(\p B_\rho)}^2+o_\delta(1).
\end{eqnarray}
Consequently, as $\delta\to0$, we obtain \eqref{7.TheBadPartGoesToZero}.
\end{proof}

Following the same lines as Proposition \ref{P7.decoupling} we obtain
\begin{equation}\label{7.boundForPhaseIndepedentondelta}
\|\n\f_\delta\|_{L^2(A_\delta)}\leq C\text{ with $C$ independent of $\delta$},
\end{equation}
and
\begin{equation}\label{7.boundForPhaseIndepedentondeltaConcverf0}
\|\n\eta_\delta\|_{L^2(A_\delta)}\to0.
\end{equation}
It follows from \eqref{7.boundForPhaseIndepedentondelta} and \eqref{7.boundForPhaseIndepedentondeltaConcverf0} that
\begin{eqnarray}\nonumber\label{7.FundEstimateAuxProbFarIncl}
I_\delta&=&\frac{d_0^2}{2}\int_{A_\delta}{|\n\theta|^2}+\frac{1}{2}\int_{A_\delta}{|\n\f_\delta|^2}
\\\nonumber
&=&\frac{d_0^2}{2}\int_{A_\delta}{|\n\theta|^2}+\frac{1}{2}\int_{A_\delta}{|\n\psi_\delta|^2}+\int_{A_\delta}{\n\psi_\delta\cdot\n\eta_\delta}+\frac{1}{2}\int_{A_\delta}{|\n\eta_\delta|^2}
\\&=&J_\delta+o_\delta(1).
\end{eqnarray}
\end{proof}
From \eqref{7.FundEstimateAuxProbFarIncl} and \eqref{7.TheBadPartGoesToZero}, we deduce that 
\begin{equation}\label{7.FundEstimateAuxProbFarInclBis}
I_{\delta}=J_{\delta}+o_\dl(1)=\pi d_0^2 \ln\frac{\rho}{\delta}+\tilde W_0(f^0) + \tilde W_1(g^0)+o_\delta(1)
\end{equation}
with
\begin{equation}\label{decoupled_energy}
\tilde W_0(f_0)=|\zeta_0|_{H^{1/2}(\d B_\rho)}^2 \text{ and } \tilde W_1(g^0) = |\phi_0|_{H^{1/2}(\p B_1)}^2.
\end{equation}
One of the main ingredients in the study of the renormalized energy is that the Dirichlet condition $f_{\rm min}(x)=\gamma_0\frac{x^{d_0}}{|x|^{d_0}}$, $\gamma_0\in\S^1$ minimizes $W_0$. More precisely, for all $f_0\in C^1(\p B_1,\S^1)$ s.t. $\deg_{\p B_1}(f_0)=d_0$, we have
\begin{equation}\label{7.RadialMinimizesFirst}
W_0(f_{\rm min})=0\leq W_0(f_0).
\end{equation}

\subsection{Energy estimates for $\S^1$-valued maps around the inclusion}\label{S7.ArroudTheInclusEst}

Let $g^0 \in C^{\infty}(\p B_1, \S^1)$ be s.t. $\deg_{\p B_1}(g^0)=d_0>0$, $ \beta_1,..., \beta_{d_0}$ are $d_0$ distinct points of $\o$,
\[
\eta_0:=\frac{1}{4}\min_i\left\{\dist( \beta_i,\p \o),\min_{j\neq i}{| \beta_i- \beta_j|}\right\}.
\]
For $r \in(0,\eta_0)$, we define
\[
\O_r:= B_1\setminus\cup_k \overline{B( \beta_k,r)},
\]
\[
\mathcal{E}_r:=\{u\in H^1(\O_r,\S^1)\,|\,\tr_{\p B_1}u=g^0\text{ and }\deg_{\p B( \beta_i,r)}(u)=1\}
\]
and
\[
\mathcal{F}_r:=\left\{u\in H^1(\O_r,\S^1)\,|\,\tr_{\p B_1}u=g^0\text{ and there are $\gamma_i\in\S^1$ s.t. }\tr_{\p B( \beta_i,r)}u(x)=\gamma_i\frac{x- \beta_i}{|x- \beta_i|}\right\}.
\]
Consider two minimization problems
\begin{equation}\tag{$R_{r}$}\label{AuxProbR}
K(r,g^0,{\bm  \beta})=K(r)=\inf_{u\in\mathcal{E}_r}{\frac{1}{2}\int_{\O_r}{a^2|\n u|^2}}
\end{equation}
and
\begin{equation}\tag{$S_{r}$}\label{AuxProbS}
L(r,g^0,{\bm  \beta})=L(r)=\inf_{u\in\mathcal{F}_r}{\frac{1}{2}\int_{\O_r}{a^2|\n u|^2}},\,{\bm  \beta}=\{ \beta_1,..., \beta_{d_0}\}.
\end{equation}
We denote $\theta=\theta_1+...+\theta_{d_0}$ where $\theta_i$ is the main argument of $\displaystyle \frac{x- \beta_i}{|x- \beta_i|}$, \emph{i.e.}, $\displaystyle\frac{x- \beta_i}{|x- \beta_i|}={\rm e}^{\imath\theta_i}$.

Let $\psi_0$ be the unique (up to an additive constant in $2\pi\Z$) solution of

\begin{equation}\label{7.EqForLimitingPhaseInsideOmegaRho}
\begin{cases}
-{\rm div}\left[a^2(\n\psi_0+\n\theta)\right]=0&\text{in }B_1
\\
{\rm e}^{\imath(\theta+\psi_0)}=g^0&\text{on }\p B_1
\end{cases}.
\end{equation}
\begin{lem}\label{L7.Ren_en_LM}{\rm (\cite{LM1}, Appendix A.)}
\[
K(r)=\frac{1}{2}\int_{\O_r}{a^2|\n\theta+\n\psi_0|^2}+\mathcal{O}(r|\ln r|),
\]
\[
L(r)=\frac{1}{2}\int_{\O_r}{a^2|\n\theta+\n\psi_0|^2}+\mathcal{O}(r|\ln r|),
\]
with
\begin{equation}\label{7.FirstPartRenEnergy}
\frac{1}{2}\int_{\O_r}{a^2|\n\theta+\n\psi_0|^2}=\pi d_0 b^2|\ln r|+\tilde W_2({\bm  \beta},g^0)+\mathcal{O}(r^2).
\end{equation}
\end{lem}
In \eqref{7.FirstPartRenEnergy}, $\tilde W_2({\bm  \beta},g^0)$, whose explicit expression is given in \cite{LM1}, formula (106), depends only on ${\bm  \beta}$ and $g^0$.

\subsection{Upper bound for the energy}\label{S7.UpBounAuxPBBBB}
\begin{lem}\label{L7;BorneStupeflip}
Fix $\rho>0$ and let $f_\v\in C^\infty(\p B_\rho)$, $f_0\in C^\infty(\p B_\rho,\S^1)$ be s.t. $f_\v\to f_0$ in $C^1(\p B_\rho)$. Let ${\bm \b}=(\b_1,...,\b_{d_0})\in \o^{d_0}$ be s.t. $\b_i\neq \b_j$ for $i\neq j$. Then, for each $g^0\in C^\infty(\p B_1,\S^1)$, the following upper bound holds:
\begin{eqnarray}\label{7.UpperBoundWithRenormalizedEnergy}
\inf_{H^1_{f_\v}(B_\rho)} F_\v\leq \pi d_0 b^2\ln\frac{b}{\xi} +\pi d_0^2 \ln\frac{\rho}{\delta}+\tilde W_0(f_0)+ \tilde W_1(g^0) + \tilde W_2({\bm  \beta},g^0) +d_0 b^2\gamma+o_\v(1).
\end{eqnarray}
Here $\tilde W_0$, $\tilde W_1$ are defined by \eqref{decoupled_energy} and $\tilde W_2$ by \eqref{7.FirstPartRenEnergy}.
\end{lem}
\begin{proof}
We construct a test function $w_\v\in H^1_{\hat{f}_{\v}}(B_{\rho/\delta},\C)$ which gives \eqref{7.UpperBoundWithRenormalizedEnergy}. Fix $0<r<\eta_0$. Let
\[
\text{$u_\delta$ be the minimizer of \eqref{7.AuxProblem} with $g^\delta\equiv g^0$ and $f^\delta=\frac{f_{\v}}{|f_{\v}|}$}
\]
and
\[
\text{$u_{r}$ be the minimizer of \eqref{AuxProbS}.}
\]
Note that $f^\delta\to f_0=\lim_\v f_\v$ in $C^1(\p B_\rho)$. For each $i=1,...,d_0$ let $u_i^{\xi,r}$ be the global minimizer of the classic Ginzburg-Landau energy in $B( \beta_i,r)$ with the parameter $\xi/b$ and the boundary condition
$u_i^{\xi,r}(x)=h_i^r(x):=\gamma_i\frac{x- \beta_i}{r}$ on $\d B( \beta_i,r)$, $\gamma_i\in\S^1$ is defined through $u_r$.
Denote
\begin{eqnarray}\label{definition_of_I}
I(\xi/b, r)&:=& \inf_{H^1_{h^r_i}(B( \beta_i, r))}\frac{1}{2}\int_{B( \beta_i,r)}{\left\{|\n u|^2+\frac{b^2}{2\xi^2}(1-|u|^2)^2\right\}}
\\\nonumber&=& \frac {1}{2}\int_{B( \beta_i,r)}{\left\{|\n u_i^{\xi,r}|^2+\frac{b^2}{2\xi^2}(1-| u_i^{\xi,r}|^2)^2\right\}}.
\end{eqnarray}
Lemma IX.1 in \cite{BBH} implies that
\begin{eqnarray}\label{7.EstimateExactClassicalProblemBall}
I(\xi/b, r) =\pi\ln\frac{b r}{\xi}+\gamma+o_\xi(1).
\end{eqnarray}
We next extend the $u_i$'s to $B_{\rho/\delta}$. For this purpose, we consider $\zeta\in C^\infty(\R,[0,1])$ s.t. $\zeta=0$ in $\R^-$ and $\zeta=1$ in $[1,\infty)$ and set
\[
\chi _\v(s{\rm e}^{\imath\theta})=\zeta\left(s-\frac{\rho}{\delta}+1\right)\left[|f_{\v}|(\rho{\rm e }^{\imath\theta})-1\right]+1.
\]
In view of \eqref{A2}, we have $\|\chi _\v-1\|_{L^2(B_{\rho/\delta})}\leq C\v$. Consider the following test function
\begin{equation}\label{7.TestFunctionUpperBound}
\tilde u=\begin{cases}
\chi _\v u_\delta&\text{in }B_{\rho/\delta}\setminus \overline{B_1}
\\
u_r&\text{in }B_1\setminus \cup \overline{B( \beta_i,r)}
\\
u_i^{\xi,r}&\text{in }B( \beta_i,r)
\end{cases}.
\end{equation}
Clearly, 
\[
\inf_{H^1_{f_{\v}}} F_{\v}\leq\hat{F}_\xi(w_\v)\leq\pi d_0^2 \ln\frac{\rho}{\delta} + \tilde W_1(g^0) + \tilde W_2({\bm  \beta},g^0) +\tilde W_0(f_0)+\pi d_0 b^2\ln\frac{b}{\xi}+d_0 b^2\gamma+o _\v(1)+h(r)
\]
$\text{with }h(r)=o_r(1)$. Thus, letting $r\to0$ as $\v\to0$ we obtain the desired upper bound.
\end{proof}

\subsection{Lower bound}\label{S7.LowBounAuxPBBBB}
We prove that the upper bound (\ref{7.UpperBoundWithRenormalizedEnergy}) is sharp by constructing the matching lower bound.
\begin{lem}
Let $\v_n\downarrow0$, $\hat v_{\v_n}$ be a minimizer of \eqref{variable_change} in \eqref{modelclassblownup} for $\v=\v_n$ and ${\bm \alpha}=(\alpha_1,...,\alpha_{d_0})\in\o^{d_0}$ be given by Theorem \ref{T7.THM3}. Denote $g_0:=\lim \tr_{\p B_1}\hat{v}_{\v }\in C^\infty(\p B_1,\S^1)$. Then the following lower bound holds:
\begin{equation}\label{7.GoodLowerAuxProbStupeflip}
F(\hat{v}_{\v }, B_{\frac{\rho}{\delta}}) \geq \pi d_0 b^2\ln\frac{b}{\xi}+\pi d_0^2 \ln\frac{\rho}{\delta}+\tilde W_0(f_0)+ \tilde W_1(g_0) + \tilde W_2({\bm  \a},g_0) +d_0 b^2\gamma+o_\v (1).
\end{equation}
\end{lem}
\begin{proof}

As in the proof of Lemma \ref{L7;BorneStupeflip}, we split $B_{\frac{\rho}{\dl}}$ into three parts: $B_{\frac{\rho}{\dl}} \setminus \overline{B_1}$, $B_1\setminus \cup \overline{B(\a_i,r)}$ and $\cup B(\a_i,r)$ with small $0<r<\eta_0$.

In $B_{\frac{\rho}{\delta}} \setminus \cup \overline{B(\a_i,r)}$ one may write $\hat{v}_{\v }=|\hat{v}_{\v }|w_\v $. Using Corollary \ref{C7.UnifCongLocNonZeroDeg} and \eqref{7.H1LocConvInLittleOmega} we have
\begin{eqnarray}\nonumber
\hat F_{\xi}(\hat{v}_{\v },B_{\frac{\rho}{\delta}}\setminus \overline{B_1})&=&\frac{1}{2}\int_{B_{\frac{\rho}{\delta}}\setminus \overline{B_1}}{\left\{\hat U_{\v }^2|\hat{v}_{\v }|^2|\n w_\v |^2+\hat U_{\v }^2|\n|\hat{v}_{\v }||^2+\frac{\hat U_{\v }^4}{2\xi^2}(1-|\hat{v}_{\v }|^2)^2\right\}}
\\\nonumber
&=&\frac{1}{2}\int_{B_{\frac{\rho}{\delta}}\setminus \overline{B_1}}{\left\{\hat U_{\v }^2|\n w_\v |^2+\hat U_{\v }^2|\n|\hat{v}_{\v }||^2+\frac{\hat U_{\v }^4}{2\xi^2}(1-|\hat{v}_{\v }|^2)^2\right\}}+o_\v (1)
\\\label{7.FondamentaleLowerBoundINTwoDomain}&\geq&\frac{1}{2}\int_{B_{\frac{\rho}{\delta}}\setminus \overline{B_1}}{\hat U_{\v }^2|\n w_\v |^2}+o_\v (1).
\end{eqnarray}
We take ${g}^{\delta}=\displaystyle\frac{\tr_{\p B_1 }\hat{v}_{\v }}{|\tr_{\p B_1 } \hat{v}_{\v }|}$ and
$f^{\delta} = \displaystyle\frac{\tr_{\p B_\rho } v_{\v }}{|\tr_{\p B_\rho } v_{\v }|}$. Note that with this choice of $f^{\delta},g^\delta$ one may
apply the results of Sections \ref{S7.AwayTheInclusEst} and \ref{S7.ArroudTheInclusEst}. From \eqref{7.FondamentaleLowerBoundINTwoDomain} we obtain the lower bound in  $B_{\frac{\rho}{\delta}} \setminus \overline{B_1}$:
\begin{equation}\label{7.estimateInFirstRegin}
\hat F_{\xi}(\hat{v}_{\v },B_{\frac{\rho}{\delta}}\setminus \overline{B_1})\geq J_{\delta}+o_\v (1)
\end{equation}
with $J_\dl$ the energy associate to the minimization problem \eqref{7.AuxProblemFixData} (see page \pageref{7.AuxProblemFixData}).

Let $v_0$ be defined by \eqref{7.H1LocConvInLittleOmega}. Since we have $v_{\v }\to v_0$ in $H^1(B_1\setminus \cup \overline{B(\a_i,r)})$  and
$\hat{U}_{\v }\to a$ in $L^2(B_1\setminus \cup \overline{B(\a_i,r)})$, from Proposition \ref{P7.equationforphistar} and Lemma \ref{L7.Ren_en_LM} we obtain 
\begin{eqnarray}\nonumber
\hat F_{\xi}(\hat{v}_{\v },B_1\setminus \cup \overline{B(\a_i,r)})&\geq&\frac{1}{2}\int_{B_1\setminus \cup \overline{B(\a_i,r)}}{a^2|\n v_0|^2}+o_\v (1)
\\\nonumber
&\geq&\frac{1}{2}\int_{B_1\setminus \cup \overline{B(\a_i,r)}}{a^2|\n \theta+\n\psi_0|^2}+o_\v (1)
\\\label{7.estimateInSecondReginBis}
&=&K(r)+\mathcal{O}(r|\ln r|)+o_\v (1),
\end{eqnarray}
where $K(r)$ is defined by \eqref{AuxProbR} (see page \pageref{AuxProbR}).

In order to complete the proof of the lemma, we need to obtain a sharp lower bound in each ball $B(\alpha_i,r)$. Actually we will prove that
\begin{equation}\label{7.EstimateInThirdRegion}
\hat F_{\xi}(\hat{v}_{\v },B(\a_i,r))\geq b^2I(\xi/b,r)+o_r(1)+o_\v (1),
\end{equation}
with $I(\xi/b,r)$ being defined in (\ref{definition_of_I}). The estimate \eqref{7.EstimateInThirdRegion} is equivalent to
\begin{equation}\label{7.EstimateInThirdRegionbis}
\hat F_{\xi}(\hat{v}_{\v },B(\a_i,r))\geq b^2I(\xi/b,r+r^2)+o_r(1)+o_\v (1).
\end{equation}
Indeed by \eqref{7.EstimateExactClassicalProblemBall} we have $I(\xi,r+r^2)-I(\xi,r)=o_r(1)$.

We now make use of the construction by Lefter and R{\u{a}}dulescu in \cite{LR0} and \cite{LR}. From Proposition \ref{P7.equationforphistar}, we know that $v_0={\rm e}^{\imath(\theta_i+\f_\star+\psi_i)}$ with $\f_\star,\,\psi_i$ harmonic, and therefore
smooth in $B(\a_i, \eta)$ ($\eta>r$ small).
Set $\sigma_i=\f_\star+\psi_i$. Without loss of generality, we can assume that $\a_i=0$ and $\sigma_i(0)=0$. Consequently, $|\sigma_i(x)|\leq C|x|$ with $C$ independent of $\eta$ and $|x|\leq\eta$. Let
\[
\hat{v}_{\v }=\lambda_\v {\rm e}^{\imath(\theta_i+\sigma_\v ^i)}\, \text{ where } \lambda_\v :=|\hat{v}_{\v }|.
\]
From Proposition \ref{P7.AsymptoticResultatNonZeroDeg} and \eqref{7.potentialBoundincompact}, we obtain that
\begin{equation}\label{7.ConvPhaseArgumentrenrm}
\sigma_\v ^i\to\sigma_i\text{ in }H^1(B_{r+r^2}\setminus \overline{B_r}),
\end{equation}
\begin{equation}\label{7.ConvModulusArgumentrenrm}
\lambda_\v \to1\text{ in }H^1(B_{r+r^2}\setminus \overline{B_r})\text{ and }\frac{1}{\xi^2}\int_{B_{r+r^2}\setminus \overline{B_r}}{(1-\lambda_\v )^2}\to0.
\end{equation}
Let
\[
\beta_\v (s{\rm e}^{\imath\theta_i})=\begin{cases}\hat{v}_{\v }(s{\rm e}^{\imath\theta_i})&\text{if }s\in[0,r)\\\displaystyle\left[\frac{1-\lambda_\v }{r^2}(s-r)+\lambda_\v \right]{\rm exp}\left\{\imath\left(\theta_i+\sigma_\v ^i\frac{-s+r^2+r}{r^2}\right)\right\}&\text{if }s\in[r,r+r^2]\end{cases}.
\]
Clearly, $\beta_\v \in H^1_{x/|x|}(B_{r+r^2})$. Consequently,
\[
b^2 I(\xi/b,r+r^2)\leq \hat F_{\xi}(\hat{v}_{\v },B_r)+ \hat F_{\xi}(\beta_\v ,B_{r+r^2}\setminus \overline{B_{r}})+o_\v (1).
\]
From \eqref{7.ConvModulusArgumentrenrm}, we easily obtain that
\[
\int_{B_{r+r^2}\setminus \overline{B_r}}{\left\{|\n|\beta_\v ||^2+\frac{1}{\xi^2}(1-|\beta_\v |)^2\right\}}=o_\v (1).
\]
It remains to estimate
\[
\int_{B_{r+r^2}\setminus \overline{B_r}}{\left|\n\left\{\theta_i+\sigma_\v ^i\frac{-s+r^2+r}{r^2}\right\}\right|^2}.
\]
From \eqref{7.ConvPhaseArgumentrenrm}
\[
\int_{B_{r+r^2}\setminus \overline{B_r}}{\left|\n\left\{\theta_i+\sigma_\v ^i\frac{-s+r^2+r}{r^2}\right\}\right|^2}=\int_{B_{r+r^2}\setminus \overline{B_r}}{\left|\n\left\{\theta_i+\sigma_i\frac{-s+r^2+r}{r^2}\right\}\right|^2}+o_\v (1).
\]
Since  $|\sigma_i(s{\rm e}^{\imath\theta})|\leq Cs$, $|\p_s \sigma_i|\leq C$ and $|\p_{\theta_i}\sigma_i|\leq C s$ we have
\begin{eqnarray}\nonumber
\left|\n\left\{\theta_i+\sigma_i\frac{-s+r^2+r}{r^2}\right\}\right|^2&=&\left|\p_s\sigma_i\frac{-s+r^2+r}{r^2}-\frac{\sigma_i}{r^2}\right|^2+\frac{1}{r^2}\left|1+\p_{\theta_i}\sigma_i\frac{-s+r^2+r}{r^2}\right|^2
\\\nonumber
&\leq&C\left[(1+r^{-2})+r^{-2}\right]=\mathcal{O}(r^{-2}).
\end{eqnarray}
Since $|B_{r+r^2}\setminus \overline{B_r}|=\mathcal{O}(r^3)$ we find that
\[
\int_{B_{r+r^2}\setminus \overline{B_r}}{\left|\n\left\{\theta_i+\sigma_\v ^i\frac{-s+r^2+r}{r^2}\right\}\right|^2}=\mathcal{O}(r).
\]
It follows that $ \hat F_{\xi}(\beta_\v ,B_{r+r^2}\setminus \overline{B_{r}})=\mathcal{O}(r)+o_\v (1)$. Consequently, \eqref{7.EstimateInThirdRegionbis} holds and thus we obtain \eqref{7.EstimateInThirdRegion}. Combining \eqref{7.estimateInFirstRegin}, \eqref{7.estimateInSecondReginBis} and \eqref{7.EstimateInThirdRegion}, together with \eqref{7.FundEstimateAuxProbFarInclBis}
and (\ref{7.FirstPartRenEnergy}), we obtain
\begin{eqnarray}\nonumber
\hat F_{\xi}(\hat{v}_{\v },B_{\frac{\rho}{\delta}})&\geq& I_{\delta}+K(r)+b^2I(\xi/b,r)+o_\v (1)+o_r(1)
\\\nonumber
&=&\pi d_0^2 \ln \frac{\rho}{\delta}+\pi d_0 b^2\ln\frac{b}{\xi}+\tilde W_0(f_0)+\tilde W({\bm \a},g_0)+
\\\label{7.LowerBOUNDTotal}&&\phantom{aaddddddaaaaaaaazzzzzzzzzzza}+d_0 b^2\gamma+o_\v (1)+o_r(1).
\end{eqnarray}
The conclusion of the Lemma follows by letting $r\to0$ as $\v\to0$.
\end{proof}
\subsection{The function $g_0$ and the points $\{\a_1,...,\a_{d_0}\}$ minimize the renormalized energy}\label{S7?MinAuxEnsdkfjh}
In the previous section, we obtained an expansion for the energy $\hat F_{\xi}(\hat{v}_{\v },B_{\frac{\rho}{\dl}})$ of the model problem.
To summarize, using \eqref{7.UpperBoundWithRenormalizedEnergy}, \eqref{7.GoodLowerAuxProbStupeflip} and Theorem \ref{T7.THM3} we get that there are $g_0=\lim \tr_{\p B_1}\hat{v}_{\v }$ and ${\bm \alpha}=(\alpha_1,...,\alpha_{d_0})\in\o^{d_0}$ s.t.
\begin{equation}\label{7.ExpressionWithRenormalizedEnergy}
\hat F_{\xi}(\hat{v}_{\v },B_{\frac{\rho}{\dl}})=\pi d_0 b^2\ln\frac{b}{\xi}+\pi d_0^2\ln\frac{\rho}{\delta}+\tilde{W}({\bm \a},g_0)+ {\tilde W}_0(f_0)+
 d_0 b^2 \g + o_\v (1),
\end{equation}
with
\begin{equation*}
\tilde{W}({\bm \a},g_0)=\tilde W_1 (g_0)+ {\tilde W_2}({\bm \a},g_0).
\end{equation*}
The goal of this section is to underline an important property of the points ${\bm \a}$, namely, that they minimize the quantity $\inf_{g^0\in C^\infty(\p B_1,\S^1)}\tilde{W}(\cdot,g^0)$.

We have the following
\begin{prop}\label{P7.MinRenEner}
Let ${\bm \b}=(\b_1,...,\b_{d_0})\in\o^{d_0}$ be a $d_0$-tuple of distinct points and let $g^0\in C^\infty(\p B_1,\S^1)$ be s.t. $\deg_{\p B_1}(g^0)=d_0$. Then
\[
\tilde W({\bm \a},g_0)\leq \tilde W({\bm \b},g^0).
\]
\end{prop}
\begin{proof}
Let $({\bm \b},g^0)$ be as in Proposition \ref{P7.MinRenEner}. Using the test function given by \eqref{7.TestFunctionUpperBound}, we obtain that for all $\v>0$ and $r>0$ (small) there is $w_\v\in H^1_{\hat f_\ve}(B_{\frac{\rho}{\dl}},\C)$ s.t.
\[
\hat F_\xi(w_\v)=\pi d_0 b^2\ln\frac{b}{\xi}+\pi d_0^2 \ln\frac{\rho}{\delta}+\tilde W({\bm \b},g^0) + \tilde W_0(f_0) +d_0
b^2\gamma+h^1_\v+h^2_r
\]
here $h^1_\v=o_\v(1)\text{ and }h^2_r=\mathcal{O}(r)$.

On the other hand, taking into account the minimality of ${\hat v}_{\v}$ and
\eqref{7.ExpressionWithRenormalizedEnergy} we have

\[
\tilde W({\bm \b},g^0) \geq \tilde W({\bm \a},g_0)+o_\v (1)+h^2_r.
\]
The previous estimate implies (letting $\v\to0$ and $r\to0$) that $\tilde W({\bm \b},g^0) \geq \tilde W({\bm \a},g_0)$ which completes the proof.
\end{proof}
Thus, for ${\bm \beta} = (\beta_1, ..., \beta_{d_0})\in\o^{d_0}$ we define
\begin{equation}\label{Ren_en_loc}
\tilde{W}({\bm \beta})=\inf_{\substack{\tilde{g}\in C^\infty(\p B_1,\S^1)\\\deg_{\p B_1}(\tilde{g})=d_0}}\tilde{W}({\bm \beta},\tilde{g})=\inf_{\substack{\tilde{g}\in C^\infty(\p B_1,\S^1)\\\deg_{\p B_1}(\tilde{g})=d_0}}\tilde{W}_1 (\tilde{g})+ {\tilde{W}_2}({\bm \beta},\tilde{g})
\end{equation}
with $\tilde{W}_1$ and $\tilde{W}_2$ given by \eqref{decoupled_energy} and \eqref{7.FirstPartRenEnergy} respectively. It follows that for ${\bm\alpha}$ given by Theorem \ref{T7.THM3} and $g_0=\tr_{\p B_1}v_0$:
\[
\tilde{W}({\bm \a})=\tilde W({\bm \a},g_0)\leq\tilde{W}({\bm \beta})\text{ for all } {\bm \beta} = (\beta_1, ..., \beta_{d_0})\in\o^{d_0}.
\]

\section{Proofs of Theorems \ref{T7.THM1} and \ref{T7.THM2}}
In this section $v_\v$ is a minimizer of $F_\v$ in $H^1_g(\O,\C)$. We split the proofs of Theorem \ref{T7.THM1} and \ref{T7.THM2} in three steps:
\begin{enumerate}[$\bullet$]
\item(Section \ref{S7.LOCationAndEstimate}) Using estimates on $|v_\v|$, we first localize the vorticity to the neighborhoods of selected inclusions. Then we find two separate energy expansions in two sub-domains of $\O$: away from the selected inclusions and around them.
\item(Section \ref{S7.ExLimSolu}) We study the asymptotic behavior of $v_{\v}$. We prove that, for small $\v$, $v_\v$ has exactly $d$ zeros of degree $1$.
\item(Section \ref{S7;MinBBHREnENErGy}) We give an expansion of $F_\v(v_\v)$ up to $o_\v(1)$ terms and relate the choice of the inclusions with vortices to the renormalized energy of Bethuel, Brezis and Hélein.
\end{enumerate}

\subsection{Locating bad inclusions}\label{S7.LOCationAndEstimate}

The following result gives a uniform bound on the modulus of minimizers away from the inclusions.
 \begin{lem}\label{L7.Lemma1}There exists $C>0$ s.t. for small $\ve$ we have
 \begin{enumerate}
\item $|v_\ve| \geq 1 - C |\ln \ve|^{-1/3}$ in $\W \setminus \cup_{i=1}^M\overline{B(a_i,  \dl)}$, 
\item there are at most $d$ points $a_{i_1}, ..., a_{i_{d'}}$ ($1\leq d'=d'_\v\leq d$) s.t.  $\{|v_\ve|<1-C |\ln \ve|^{-1/3}\}\subset\cup_{k=1}^{d'}B(a_{i_k},  \dl)$.
\end{enumerate}
\end{lem}
\begin{proof}
Using Lemma \ref{L7.BadDiscsLemma} with $\displaystyle\chi=|\ln \ve|^{-1/3}$, we obtain that there exist $C,C_1>0$ s.t. for $\v>0$ small,
\begin{center}
if $\displaystyle F_\v(v_\v,B(x,\v^{{1/4}}))< |\ln \ve|^{\frac{1}{3}}-C_1$ then $|v_\v|\geq 1-C\chi$ in $B(x,\v^{1/2})$.
\end{center}
We prove 1. by contradiction. Assume that, up to a subsequence, there is $ x_{\v } \in \W \setminus \cup_{i=1}^{M} \overline{B(a_i,  \delta)}$, s.t. $|v_{\v }(x_{\v })| < 1 - C |\ln \v |^{-1/3}$ with $C$ given by Lemma \ref{L7.BadDiscsLemma}.
From Lemma \ref{L7.BadDiscsLemma}  and Proposition \ref{P7.UepsCloseToaeps}
\begin{equation}\label{7.TOproveLocationofvorticityInMainProblems}
\frac{1}{2}\int_{B(x_\v,\v ^{1/4})}{\left\{|\n v_{\v }|^2+\frac{1}{2\v^2}(1-|v_{\v }|^2)^2\right\}}\geq |\ln \v |^{1/3} - \mathcal{O}(1).
\end{equation}
Fix a bounded, simply connected domain $\O'$ such that $\overline{\O} \subset \O'$, and extend $v_{\v }$ by a fixed smooth $\S^1$-valued map $v$ in $\O' \setminus \overline{\O}$, s.t. $v = g$ on $\d \O$.

In view of \eqref{7.UpperboundAuxPb} for Case I or \eqref{7.UpperboundAuxPbCaseII} for Case II, there exists $\tilde C > 0$ s.t. for small $\v$
\[
\frac{1}{2}\int_{\O'}{|\n v_{\v }|^2+\frac{1}{2\v ^2}(1-|v_{\v }|^2)^2}\leq \tilde C|\ln\v |.
\]
Therefore, the map $v_{\v }$ in $\O'$ satisfies the condition of Theorem 4.1 \cite{SS1}. This theorem guarantees that
\begin{enumerate}[$\bullet$]
\item there exists $\mathcal{B}^\v = \{B_j^\v\}$, a finite disjoint covering of the set
\[
\{x\in\O' \, | \,\dist(x,\p\O')>\v/b\text{ and } |v_{\v }(x)| < 1 - (\v /b)^{1/8}\},
\]
\item such that ${\rm rad} (\mathcal{B}^\v) := \sum_{j} {\rm rad} (B_j^\v)  \leq 10^{-2}\cdot\dist(\o,\p B(0,1))\cdot\delta$,
\item and, denoting $d_j=|\deg_{\p B_j}(v_{\v })|$ if $B_j^\v\subset\{\dist(x,\p\O')>\v /b\}$ and $d_j=0$ otherwise, we have
\begin{eqnarray}\nonumber
\frac{1}{2}\int_{\cup B_j^\v}{|\n v_{\v }|^2+\frac{b^2}{2\v ^2}(1-|v_{\v }|^2)^2} &\geq& \pi \sum_jd_j \ln\frac{\delta}{\v }-C
\\\label{7.LowerBoundOnSandierSerfatyCovering}&=&\pi \sum_jd_j |\ln\xi|-C,
\end{eqnarray}
with $C$ independent of $\v $.
\end{enumerate}
Note that since $|v_{\v }|\equiv1$ in $\O'\setminus\overline{\O}$, if $d_j\neq0$ then $B^\v_j\subset\{\dist(x,\p\O')>\v /b\}$. Consequently, we have $d_j=|\deg_{\p B^\v_j}(v_{\v })|$.

Assertion 1. follows as in the proof of Proposition \ref{P7.ConvOfBadDiscsTo0} (use \eqref{7.TOproveLocationofvorticityInMainProblems}, \eqref{7.LowerBoundOnSandierSerfatyCovering} instead of \eqref{7.ContradictionForUnifConv} and \eqref{7.LowerBoundOnSandierSerfatyCovering1}).

The proof of Assertion {\it 2.} of Lemma \ref{L7.Lemma1} goes along the same lines.
\end{proof}
We next obtain the following lower bounds for the energy.
\begin{lem}\label{L7.Lemma2}
For $k\in\{1,...,d'\}$, we denote $d_k=d_k^\v  = \deg_{\d B(a_{i_k},  \dl)} (v_\v )$. There exist $C,\eta_0 > 0$ s.t. for small $\v$ and $\rho\in[2\delta,\eta_0]$ we have
\begin{equation}\label{bound1}
\frac{1}{2} \int_{\O \setminus \cup_{k=1}^d \overline{B(a_{i_k}, \rho)}} |\nabla v_\v|^2  \geq \pi \sum_{k=1}^d d_k^2 |\ln \rho| - C
\end{equation}
and
\begin{equation}\label{bound3}
F_\v(v_\v, B(a_{i_k}, {2 \dl})) \geq \pi |d_k| b^2 |\ln \xi| - C.
\end{equation}
\end{lem}
\begin{proof}Let $\eta_0=10^{-2}\min_i\left\{\dist(a_i,\p\O),\min_{j\neq i}|a_i-a_j|\right\}$ and $0<\rho<\eta_0$.

We prove \eqref{bound1}. By Lemma \ref{L7.Lemma1}, $|v_{\v}|\geq1/2$ in $\O \setminus \cup_{k=1}^d \overline{B(a_{i_k}, \rho)}$, therefore, $w_{\v} =  \frac{v_{\v}}{|v_{\v}|}$ is well-defined in this domain. From direct computations in $B(a_{i_k}, \eta_0)\setminus\overline{B(a_{i_k}, \rho)}$  we have
\begin{equation}\label{for_S1}
\frac{1}{2} \int_{\O \setminus \cup_{k=1}^d \overline{B(a_{i_k}, \rho)}} |\nabla w_{\v }|^2  \geq \pi \sum_{i=1}^d d_i^2 \ln \frac{\eta_0}{\rho}.
\end{equation}
We claim that the bound (\ref{bound1}) holds with $C = |\ln\eta_0| + 1$. Argue by contradiction: assume that up to a subsequence we have:
\begin{equation}\label{assumption}
\frac{1}{2} \int_{\O \setminus \cup_{k=1}^d \overline{B(a_{i_k}, \rho)}} |\nabla v_{\v }|^2  \leq \pi \sum_{i=1}^d d_i^2 \ln \frac{\eta_0}{\rho}  - 1.
\end{equation}
On the other hand, we have
\[
|\nabla v_{\v }|^2 = |v_{\v }|^2 |\nabla w_{\v }|^2 + |\nabla |v_{\v }||^2
\]
and therefore
\begin{equation}\label{aux}
\int_{\O \setminus \cup_{k=1}^d \overline{B(a_{i_k}, \rho)}} |\nabla v_{\v }|^2 \geq \int_{\O \setminus \cup_{k=1}^d \overline{B(a_{i_k}, \rho)}} |\nabla w_{\v }|^2  -  (1-|v_{\v }|^2) |\nabla w_{\v }|^2.
\end{equation}
Using the fact that $|v_{\v }| \geq \frac{1}{2}$ in $\O \setminus \cup_{k=1}^d \overline{B(a_{i_k}, \rho)}$ we see that $|\nabla w_{\v }| \leq 2 |\nabla v_{\v }|$.
Therefore, by (\ref{assumption}), (\ref{7.MainHyp}) and Lemma \ref{L7.Lemma1} we estimate the last term in (\ref{aux}):
\begin{equation}\label{last_term}
\int_{\O \setminus \cup_{k=1}^d \overline{B(a_{i_k}, \rho)}} (1-|v_{\v }|^2) |\nabla w_{\v }|^2 \leq
C |\ln {\v }|^{-\frac{1}{3}} \int_{\O}|\nabla v_{\v }|^2 \leq C \frac{|\ln \rho|}{|\ln {\v }|^{\frac{1}{3}}} \to 0.
\end{equation}
By combining (\ref{for_S1}), (\ref{aux}) and (\ref{last_term}), we see that \eqref{assumption} cannot hold for small $\v$; this implies (\ref{bound1}).\\
We now prove \eqref{bound3}. Performing the rescaling $\hat x = \frac{x - a_{i_k}}{\dl}$, we obtain
\[
F_\v(v, B(a_{i_k}, 2 \dl)) = \hat F_{\xi}(\hat v, B(0,2)) = \frac{1}{2} \int_{B(0,2)} \left\{\hat U_{\ve}^2 |\n \hat v|^2 +
\frac{1}{2 \xi^2}\hat U_{\ve}^4(1-|\hat v|^2)^2 \right\}\, {\rm d} \hat x,
\]
where, as in the model problem we set $\hat v (\hat x) = v( \dl \hat x)$ and $\xi =\displaystyle \frac{\v}{\delta}$.

By Theorem 4.1 \cite{SS1}, for $r= 10^{-2}$ there are $C>0$ and a finite covering by disjoint balls $B_1,...,B_N$
(with the sum of radii at most $r$) of $\{\hat x\in B(0,2-\xi/b) \,|\,1-|\hat v_{\v}(\hat x)|\geq(\xi/b)^{1/8}\}$ and
\begin{equation}\label{7.EnergyOfTheHoleAux}
\frac{1}{2}\int_{\cup_jB_j}{\left\{|\n \hat v_{\v}|^2+\frac{b^2}{2\xi^2}(1-|\hat v_{\v}|^2)^2\right\}}\geq\pi D_k|\ln\xi|-C,
\end{equation}
$D_k=\sum_{j}|m_j|$ and
\[
m_j=\begin{cases}\deg_{\p B_j}(\hat v_{\v})&\text{if } {\rm dist}(B_j, \d B(0,2)) \geq \xi/b \\0&\text{otherwise}\end{cases}.
\]
Since, by Lemma \ref{L7.Lemma1}, $|\hat v_\v|\geq1/2$ in $B(0,2) \setminus \overline{B(0,1)}$, $D_k \geq d_k$, and (\ref{bound3}) follows from \eqref{7.EnergyOfTheHoleAux} and the estimate $U_\v\geq b$.

\end{proof}

\begin{cor}\label{C7.cor1}
Assume that $M\geq d$. Then $d'=d$ and $d_k=1$ for each $k$.
\end{cor}
\begin{cor}\label{C7.cor2}
Assume that $M< d$. Then $d'=M$ and $d_k\in\left\{\left[\dfrac{d}{M}\right],\left[\dfrac{d}{M}\right]+1\right\}$ for each $k$.
\end{cor}
\begin{proof}[Proof of Corollaries \ref{C7.cor1} and \ref{C7.cor2}]

By combining (\ref{bound1}) and (\ref{bound3}) we obtain the lower bound for $F_\ve$ in $\O$:
\begin{equation}\label{lowerbnd_in_OmStupeflip}
F_\v(v_\ve)+C_1 \geq \pi \sum_{i=1}^{M} \left\{\deg_{\p B(a_i,\delta)}(v_\v)^2|\ln \delta| + b^2 |\deg_{\p B(a_i,\delta)}(v_\v)| |\ln\xi|\right\}  .
\end{equation}
The conclusions of the above corollaries are obtained by solving the discrete minimization problem of the RHS of \eqref{lowerbnd_in_OmStupeflip}.
\end{proof}



As a direct consequence of Proposition \ref{P7.UpperboundAuxPb} and Lemma \ref{L7.Lemma2}, we have
\begin{cor}\label{C7.Lemma3}
There is $C>0$ independent of $\v$ s.t. for $1> \rho > 2 \dl$ we have
\begin{eqnarray*}
\frac{1}{2}\int_{\Omega \setminus \cup_{k=1}^d \overline{B(a_{i_k}, \rho)}} |\nabla v_\v|^2 \, {\rm d}x &=&\pi\sum_{k=1}^{d'}d_k^2|\ln\rho|+\mathcal{O}(1)
\\&=&\begin{cases}  \pi d |\ln \rho| + \mathcal{O}(1)&\text{in Case I}\\\pi\displaystyle\min_{\substack{\tilde{d}_1,...,\tilde{d}_M\in\Z\\\tilde{d}_1+...+\tilde{d}_M=d}}\sum_{i=1}^{M}\tilde{d}_i^2|\ln\rho|+\mathcal{O}(1)&\text{in Case II}\end{cases}
\end{eqnarray*}.
\end{cor}

\subsection{Existence of the limiting solution}\label{S7.ExLimSolu}

We now return to the proof of Theorems \ref{T7.THM1} and \ref{T7.THM2}. 

Recall that $\{i_1^\v,...,i_{d'}^\v\}$ is a set of distinct elements of $\{1,...,M\}$. We choose $\v$ small enough so that $i_j$'s are independent of $\v$, thus we may simply denote this set by $\{i_1,...,i_{d'}\}$. In Case I, we have $d'=d$ and we may assume that $\{i_1,...,i_{d'}\}=\{1,...,d\}$. In Case II, we have $d'=M$.

Lemma \ref{L7.Lemma1} and Corollary \ref{C7.Lemma3} imply that for an appropriate extraction $\v=\v_n\downarrow0$ and for a compact $K \subset \Omega \setminus \{a_{i_1},...,a_{i_{d'}}\}$, there is $C_K>0$ s.t. for small $\v$ we have
\begin{equation*}
F_{\v }(v_{\v }, K) \leq C_K
\end{equation*}
and
\[
|v_{\v }(x)| \geq 1-C|\ln\v |^{-1/3}\, \text{ for all } x \in  K.
\]

 Therefore, when  $\v \to0$, up to a subsequence, there exists $ v^* \in H^1(\O \setminus \{a_{i_1},...,a_{i_{d'}}\}, \S^1)$ s.t.
 $v_{\v } \weak v^* \in H^1_{\rm loc}(\Omega \setminus \{a_{i_1},...,a_{i_{d'}}\})$. \\

We now fix such sequence and a compact $K \subset \O \setminus \{a_{i_1},...,a_{i_{d'}}\}$. If $K \subset \Omega \setminus \{a_i, 1 \leq i \leq M\}$, then we have $K \cap \o_\dl = \emptyset$ for small $\v$. By exactly the same argument as in Proposition \ref{P7.SmoothBound} we deduce that $v_{\v }$ is bounded in $C^k(K)$ for all $k \geq 0$ and $1-|v_{\v }|^2 \leq C_{K} \v ^2$ in $K$.

Consequently, up to subsequence we have for a compact set $K \subset \Omega \setminus \{a_1,...,a_M\}$
\begin{equation}\label{unif_conv}
v_{\v } \to v^* \text{ in } C^k(K) \text{ and } 1-|v_{\v }|^2 \leq C_{K} \v ^2.
\end{equation}
Now, assume that $K$ is s.t. $K \subset \O \setminus \{a_{i_1},...,a_{i_{d'}}\}$ but $K \cap \o_\dl \neq \emptyset$ (then we are in Case I). Without loss of generality, assume
$K = \overline{B(a_{k_0}, R)}$, where $a_{k_0} \in \{a_{d+1},...,a_M\}$ and $R>0$ is sufficiently small in order to have $K \cap \{a_1,...,a_M\} =\{a_{k_0}\}$.\\

Let $h_\v  : = {\rm tr}_{\d K} v_{\v }$. Since $\d K \subset \O \setminus \{a_1,...,a_M\}$, we have $h_\v  \to h_0$ in $C^{\infty}(\d K)$ (possibly after passing to a subsequence). Since ${\rm deg}(h_\v , \d K) = 0$ we have ${\rm deg}(h_0, \d K) = 0$ and consequently there is some $\f_0\in C^\infty(\p K,\R)$ s.t. $h_0= e^{i \f_0}$.\\

Let $\tilde v$ be a minimizer of $\displaystyle\int_{K} |\nabla v|^2$ in the class $ H_{h_0}^1(K, \S^1)$. Clearly,
\[
\int_{K}|\nabla \tilde v|^2 \leq \int_{K}|\nabla v^*|^2.
\]
On the other hand, since $U_\v \leq 1$, we may construct (in the spirit of \cite{BBH1}) a test function and find that (see formula (93) in \cite{BBH1})
\begin{equation}\label{93}
F_{\v }(v_{\v }, K)  \leq  \frac{1}{2} \int_{K} |\nabla \psi_{\v }|^2 + C \v ,
\end{equation}
where $\psi_{\v }$ is the solution of
\[
\begin{cases}
\Delta \psi_\v  = 0 & \text{in } K\\
\psi_\v  = \f_\v  &\text{on }  \d K
\end{cases}.
\]
Here, $\f_\v $ is defined by
\[
e^{i \f_\v } = \frac{h_\v }{|h_\v |} \ \text{on} \ \d K.
\]
As $\v\to0$, we have
\begin{equation}\label{str_conv}
\psi_\v  \to \psi_0 \text{ strongly in } H^1(K), \text{ where }
\begin{cases}
\Delta \psi_0 = 0& \text{in} \ K
\\
\psi_0 = \f_0 & \text{on }  \d K
\end{cases}.
\end{equation}
From the fact that $v_{\v }\weak v_*$ in $L^2(K)$, $U_{\v }\to1$ in $L^2(K)$ and $|U_{\v }|\leq1$ we have $U_{\v }^2\n v_{\v }\weak v^*$ in $L^2(K)$. Consequently, we obtain
\begin{eqnarray}\label{semicont}
\frac{1}{2} \int_{K}|\nabla v^*|^2 &\leq& \liminf_{\v\to0} \frac{1}{2}\int_{K}U_{\v }^2|\nabla v_{\v }|^2 
 \leq \liminf_{\v\to0} F_{\v }(v_{\v }, K).
\end{eqnarray}
Combining (\ref{93}), \eqref{str_conv} and (\ref{semicont}) we deduce that
\[
\int_{K}|\nabla v^*|^2 \leq \int_{K} |\nabla \psi_0|^2 = \int_{K} |\nabla \tilde v|^2.
\]
It follows that $v^*$ minimizes the Dirichlet functional in
\[
H^1_{h_0}(K, \S^1) := \{v \in H^1(K, \S^1), v = h_0 \text{ on } \d K\}.
\]
We find that  hence $\tilde v = v^*$ in $K$. By a classic result of Morrey \cite{Mor} (see also \cite{BBH1}), $v^*$ satisfies (\ref{harmonic_map}). Moreover, as follows from weak lower semicontinuity of Dirichlet integral, (\ref{93}), (\ref{str_conv}) and \eqref{semicont}
\[
\frac{1}{2} \int_{K}|\nabla v^*|^2 \leq \liminf_{\v\to0} \frac{1}{2}\int_{K}|\nabla v_{\v }|^2 \leq \limsup_{\v\to0} F_{\v }(v_{\v },K) \leq  \frac{1}{2} \int_{K}|\nabla v^*|^2.
\]
Therefore,
\begin{equation}\label{H1StrongConvCompaStup}
\text{ $v_{\v }$ converges to $v^*$ strongly in $H^1(K)$.}
\end{equation}

From \eqref{unif_conv} and \eqref{H1StrongConvCompaStup} we obtain that $v_{\v }\to v^*$ in $H^1_{\rm loc}(\O\setminus\{a_1,...,a_{d'}\})$. The convergence up to $\p\O$ will be established in the next section.

In order to prove Assertion {\it 3.} of Theorem \ref{T7.THM1} and  Assertion {\it 2.} of Theorem \ref{T7.THM2}, note that, for small $\rho > 0$, the estimate (\ref{unif_conv}) implies that $f_{\v } :={\rm tr}_{\d B(a_{i_k}, \rho)} v_{\v }$  satisfies the conditions (\ref{A1}) and (\ref{A2}) of Theorem \ref{T7.THM3}. This gives us {\it 3.} of Theorem \ref{T7.THM1} and  {\it 2.} of Theorem \ref{T7.THM2}.

Assertion {\it 3.} of Theorem \ref{T7.THM2} is is a consequence of Corollary \ref{C7.cor2}.

\subsection{The macroscopic position of vortices minimizes the Bethuel-Brezis-Hélein renormalized energy}\label{S7;MinBBHREnENErGy}
Let us recall briefly the concept of the renormalized energy $W_g((b_1,d_1),...,(b_k,d_k))$ with
\[
\begin{cases}g\in C^\infty(\p\O,\S^1)\text{ s.t. }\deg_{\p\O}(g)=d\\b_1,...,b_k\in\O,\,b_i\neq b_j\text{ for }i\neq j\\d_i\in\Z\text{ and }\sum_i d_i=d\end{cases}.
\]
For small $\rho>0$, consider $\O_\rho=\O\setminus\cup_i\overline{ B(b_i,\rho)}$ and the minimization problem
\[
I_\rho((b_1,...,b_k),(d_1,...,d_k))=\inf_{\substack{w\in H^1(\O_\rho,\S^1)\text{ s.t.}\\w=g\text{ on }\p\O\\w(b_i+\rho{\rm e}^{\imath\theta})=\alpha_i{\rm e}^{\imath d_i\theta},\,\alpha_i\in\S^1}}\frac{1}{2}\int_{\O_\rho}|\n w|^2.
\]
Such problem is studied in detail in \cite{BBH} (Chapter 1). In particular Bethuel, Brezis and Hélein proved that for small $\rho$, we have
\[
I_\rho((b_1,...,b_k),(d_1,...,d_k))=\pi d|\ln\rho|+W_g((b_1,d_1),...,(b_k,d_k))+o_\rho(1).
\]
This equality plays an important role in the study done in \cite{BBH}. In the minimization problem of the classical Ginzburg-Landau functional
\[
\frac{1}{2}\int_\O\left\{|\n u|^2+\frac{1}{2\v^2}(1-|u|^2)^2\right\}\text{, $u\in H^1_g$,}
\]
 the vortices (with their degrees) of a minimizer tend to form (up to a subsequence) a minimal configuration for $W_g$.

 We prove in this section that the (macroscopic) location of the vorticity of minimizers of $F_\v$ is related to the minimization problem of $W_g((b_1,...,b_k),(d_1,...,d_k))$ with $b_1,...,b_k\in\{a_1,...,a_M\}$.

We present here the argument for Case I (Theorem \ref{T7.THM1}). The argument in Case II is analogous.

The proof of Assertion {\it 4.} relies on two lemmas, providing sharp upper and lower bounds.
 \begin{lem}\label{L7.Lemma4}
There exists $\rho_0 > 0$ s.t., for every $\rho < \rho_0$ and every $\v > 0$, we have
\begin{equation}\label{uprbndlemma}
F_\v(v_\v) \leq \pi d |\ln\rho| + d J(\v, \rho) + W_g((a_{i_1},1), ..., (a_{i_d},1)) + o_\rho(1),
\end{equation}
where $J(\ve, \rho) = \inf_{u \in H^1_{g_\rho}(B_\rho(0))} F_\ve(u)$ with $g_\rho = \frac{z}{\rho}$ on $\d B(0, \rho)$.
\end{lem}
\begin{proof}
The proof, via construction of a test function, is the same as proof of Lemma VIII.1 in \cite{BBH}.
\end{proof}
\begin{lem}\label{L7.Lemma5}
Let $\rho > 0$, $\rho < \rho_0$. Then for small $\v$ we have
\begin{equation}\label{lwbndIrho}
F_\ve(v_\ve) \geq \pi d |\ln\rho| + d J(\ve, \rho) + W_g((a_{i_1},1), ..., (a_{i_d},1)) + o_\rho(1).
\end{equation}
\end{lem}
\begin{proof}
Split the domain $\O$ into two sub-domains: $\O \setminus \cup_{i} \overline{B(a_{k_i}, \rho)}$ and $\cup_{i} B(a_{k_i}, \rho)$.
We start with the lower bound in the first sub-domain. By the previous estimate, $v_{\v }$ weakly converges to $v^*$ in
$H^1(\O \setminus \cup_{i} \overline{B(a_{k_i}, \rho)})$. This implies that
\[
\liminf\frac{1}{2} \int_{\O \setminus \cup_{k} \overline{B(a_{i_k}, \rho)}} U_{\v }^2|\nabla v_{\v }|^2 \geq \frac{1}{2} \int_{\O \setminus \cup_{k} \overline{B(a_{i_k}, \rho)}} |\nabla v^*|^2.
\]
Here, we used the fact that, since $U_{\v }\to1$ in $L^2(\O)$, $|U_{\v }|\leq1$ and $\n v_{\v }\weak \n v^*$ in $L^2(\O \setminus \cup_{k} \overline{B(a_{i_k}, \rho)})$, we have $U_{\v }\n v_{\v }\weak \n v^*$ in $L^2(\O \setminus \cup_{k} \overline{B(a_{i_k}, \rho)})$.

Thus we deduce that, for small $\v$,
\begin{equation}\label{x}
\frac{1}{2} \int_{\O \setminus \cup_{k} \overline{B(a_{i_k}, \rho)}} U_{\v }^2|\nabla v_{\v }|^2 \geq \frac{1}{2} \int_{\O \setminus \cup_{k} \overline{B(a_{i_k}, \rho)}} |\nabla v^*|^2 -
\rho^2.
\end{equation}
On the other hand, as proved in \cite{BBH},
\begin{equation}\label{y}
\frac{1}{2} \int_{\O \setminus \cup_{k} \overline{B(a_{i_k}, \rho)}} |\nabla v^*|^2  \geq \pi d \ln{\frac{1}{\rho}} + W_g((a_{i_1},1),...,(a_{i_d},1)) + o_\rho(1).
\end{equation}
Thus, combining \eqref{x}, \eqref{y} and using Proposition \ref{P7.UepsCloseToaeps}, for $\v$ sufficiently small, we have
\begin{equation}\label{LBoutside}
F_{\v }(v_{\v }, \O \setminus \cup_{k} \overline{B(a_{i_k}, \rho)}) \geq \pi d \ln{\frac{1}{\rho}} +  W_g((a_{i_1},1),...,(a_{i_d},1)) + o_\rho(1).
\end{equation}
By Theorem \ref{T7.THM4} and Corollary \ref{C7.cor1} we have the following energy expansion:
\begin{equation}\label{c}
F_\ve(v_\v, B(a_{i_k}, \rho)) = \pi \ln \rho + \pi b^2 |\ln \v| + \pi (1-b^2) |\ln \dl| + \tilde W({\bm\alpha}) + \tilde W_0(f_0) + b^2 \gamma + o_{\v}(1).
\end{equation}
Similarly, applying Theorem \ref{T7.THM4} to $J(\v,\rho)$ we obtain
\begin{equation}\label{z}
J(\v, \rho) = \pi \ln \rho + \pi b^2 |\ln \v| + \pi (1-b^2) |\ln \dl| + \tilde W({\bm\alpha}) + \tilde W_0(z/|z|) + b^2 \gamma + o_{\v}(1).
\end{equation}
Here, the local renormalized energy $ \tilde W({\bm \a})$ is given by \eqref{Ren_en_loc} and is the same in \eqref{c} and \eqref{z}.

From \eqref{7.RadialMinimizesFirst},  $\tilde W_0(f_0) \geq 0$ while $\tilde W_0(\frac{z}{|z|}) = 0$. Consequently, we have $F_\ve(v_\v, B(a_{i_k}, \rho)) - J(\v, \rho) \geq o_\v(1)$. Hence $\forall \rho > 0$ there exists $\v_\rho > 0$ s.t. for $\v < \v_\rho$ we have
\begin{equation*}
F_\v(v_\v, B(a_{i_k}, \rho)) \geq J(\v,\rho) - \rho^2
\end{equation*}
and thus
\begin{equation}\label{a}
F_\v(v_\v, \cup_k B(a_{i_k}, \rho)) \geq d J(\v,\rho) -d\rho^2,
\end{equation}
which gives the lower bound in the second sub-domain. From \eqref{LBoutside} and \eqref{a} the bound (\ref{lwbndIrho}) follows.
\end{proof}
Combining Lemma \ref{L7.Lemma4} and Lemma \ref{L7.Lemma5}, we see that the points $\{a_{i_k}, 1 \leq k \leq d\}$ minimize $W_g$ among $a_1, ..., a_M$.
The expansion (\ref{expansion}) follows from (\ref{uprbndlemma}), (\ref{lwbndIrho}) and (\ref{z}).

We next turn to convergence of $v_\v$ up to the boundary. It suffices to prove the $H^1$-convergence of  $v_{\v }$ in $\O_\rho=\O\setminus\cup_m \overline{B(a_{i_m},\rho)}$ (for small $\rho>0$). We argue by contradiction and we assume that there are some $\rho_1>0$ and $\eta>0$ s.t.
\begin{equation}\label{TOGETSTRONGCONVUNTILTHEBOUNDARYSTUPSTUP}
\liminf\frac{1}{2}\int_{\O_{\rho_1}}|\n v_{\v }|^2\geq\frac{1}{2}\int_{\O_{\rho_1}}|\n v^*|^2+\eta.
\end{equation}
Note that for all $\rho\leq\rho_1$, \eqref{TOGETSTRONGCONVUNTILTHEBOUNDARYSTUPSTUP} still holds in $\O_\rho$.

If, in the proof of Lemma  \ref{L7.Lemma5}, we replace \eqref{x} by \eqref{TOGETSTRONGCONVUNTILTHEBOUNDARYSTUPSTUP} (with $\rho_1$ replaced by $\rho$), then we obtain for small $\rho$ a contradiction with Lemma \ref{L7.Lemma4}. The proof of Theorem \ref{T7.THM1} is complete. The last assertion of Theorem \ref{T7.THM2} is obtained along the same lines.



\appendix
\section{Proof of Proposition \ref{P7.UepsCloseToaeps}}\label{S.AppendixEstimateSpecialSol}

Let $x_0\in V_{R}$ be s.t. $B_R = B(x_0, R) \subset \O \setminus \overline{\w_\delta}$ and assume that $x_0=0$. 

We follow the proof of Lemma 2 in \cite{BBH1}.

In $B_R$, $\eta = 1 - U_\v$ satisfies
\begin{equation*}
\left\{\begin{array}{cl}-\v^2\Delta\eta+ t \eta=-\eta(\eta^2-3\eta+2-t)&\text{in }B_R\\\eta\leq1&\text{on }\p B_R\end{array}\right.,
\end{equation*}
here, $t$ will be chosed later.

Since $\eta\in(0,1-b)$, if we take $t=b(1+b)$, then we have
\[
-\v^2\Delta\eta+t\eta\leq0 \text{ in }B_R.
\]
On the other hand, the function $w(x)= e^{\gamma(|x|^2-R^2)}$ satisfies
\[
\left\{\begin{array}{cl}-\v^2\Delta w+ t w=\left[-4\v^2\gamma(1+\gamma|x|^2)+t \right]w&\text{in }B_R\\w=1&\text{on }\p B_R\end{array}\right..
\]
A simple computation gives that
\[
\left\{\begin{array}{c}-\v^2\Delta w+t w\geq0\text{ in }B_R\\\gamma>0\end{array}\right.\Leftrightarrow0<\gamma\leq\frac{-\v+\sqrt{\v^2+t R^2}}{2R^2\v}.
\]
Take
\[
\gamma=\frac{-\v+\sqrt{\v^2+t R^2}}{2R^2\v}>0.
\]
Setting $v=\eta-w$, we have
\[
\left\{\begin{array}{cl}-\v^2\Delta v+t v\leq0&\text{in }B_R\\v\leq0&\text{on }\p B_R.\end{array}\right.
\]
By the maximum principle, we have $v\leq0$ in $B_R$. Therefore,
\[
\eta(0)\leq \text{exp}\left\{-\frac{-\v+\sqrt{\v^2+t R^2}}{2\v}\right\}\leq C e^{-\frac{\sqrt{t}R}{4\v}}.
\]
Consequently, (\ref{7.UepsCloseToaeps}) holds in $\{x \in \O\,|\, {\rm dist}(x, \d \O) \geq R, {\rm dist}(x, \w_\delta) > R\}$. The estimate close to the $\p\O$ is a direct consequence of $0\leq U_\v\leq1$, (\ref{7.UepsCloseToaeps}) holds in $\{x \in\O\,|\, {\rm dist}(x, \d \O) \geq R, {\rm dist}(x, \w_\delta) > R\}$ and the equation $-\Delta U_\v=\dfrac{1}{\v^2}U_\v(1-|U_\v|^2)$ in $\{x \in \O\,|\,  {\rm dist}(x,  \w_\delta) > R\}$.
Using a similar argument, we establish (\ref{7.UepsCloseToaeps}) in the case  $V_R \cap \w_\delta$. The proof of (\ref{7.UepsCloseToaeps}) is complete.\\

In order to prove (\ref{P7.GradUepsCloseToaeps}), note that in $W_R := \{ x \in \O \,|\, {\rm dist}(x, \p \o_\delta) \geq R, {\rm dist}(x, \p \O) \geq R\}$
the function $\eta=a_\v-U_\v$ satisfies $\Delta \eta= \frac{U_\v}{\ve^2}(a_\v^2 - U_\v^2)$. Thus, applying Lemma A.1 \cite{BBH1} to $\eta$ in conjunction with \eqref{7.UepsCloseToaeps} and the fact that $R\geq\v$, we obtain 
\begin{equation*}
|\nabla \eta| \leq \frac{C_1 e^{-\frac{c R}{\v}}}{\v}  \text{ in }W_R.
\end{equation*}
Thus \eqref{P7.GradUepsCloseToaeps} holds far away from $\p\O$ and the inclusions.

We next prove that the bound (\ref{P7.GradUepsCloseToaeps}) holds near $\d \O$.

Indeed, fix a smooth compact $K \subset \O$ s.t. for small $\delta$ we have $\o_\delta\subset K$. Clearly, by (\ref{7.UepsCloseToaeps}), $0\leq\eta_K :={\rm tr}_{\d K} \eta \leq Ce^{-\frac{c R}{\v}}$. In $\O \setminus K$, $\eta$ satisfies
\[
\begin{cases}
\Delta \eta = \frac{1}{\ve^2}U(1 + U) \eta& \text{in }\O \setminus K\\
\eta = 0 &\text{on } \d \O\\
\eta = \eta_K &\text{on } \d K
\end{cases}.
\]
Let $\eta = \eta_1 + \eta_2$ be s.t. $\eta_1$ solves
\[
\begin{cases}
\Delta \eta_1 = \frac{1}{\ve^2}U(1 + U) \eta &\text{in }\O \setminus K\\
\eta_1 = 0 &\text{on } \d \O \cup \d K
\end{cases}
\]
and $\eta_2$ satisfies
\[
\begin{cases}
\Delta \eta_2 = 0 &\text{in }\O \setminus K\\
\eta_2 = 0 &\text{on } \d \O\\
\eta_2 = \eta_K &\text{on } \d K
\end{cases}.
\]
Note that $\|\eta_2\|_{L^\infty}\leq C e^{-\frac{c R}{\v}}$ and thus $\|\eta_1\|_{L^\infty}\leq C e^{-\frac{c R}{\v}}$.

Lemma A.2 in \cite{BBH1} implies the existence of a constant $C_{\O\setminus K}>0$ s.t.
\[
|\nabla \eta_1| \leq \frac{C_ {\O\setminus K} e^{-\frac{c R}{\v}}}{\v}  \text{ in } {\O\setminus K}.
\]
In order to estimate $\n\eta_2$ near $\p\O$, we express $\eta_2$ in terms of  Green's function $G(x,y)$ in $\O\setminus K$:
function, \emph{i.e.}
\begin{equation}\label{Greens}
\eta_2(x) = -\int_{\d K} \eta_K(y) \frac{\d G}{\d \nu}(x,y) d S(y).
\end{equation}
It follows from (\ref{Greens}) and (\ref{7.UepsCloseToaeps}) that
$|\nabla \eta_2| \leq C_0 e^{-\frac{c R}{\v}}$ away from $\d K$. The estimate \eqref{P7.GradUepsCloseToaeps} is proved.

\section{Proof of Proposition \ref{P7.UpperboundAuxPb}}\label{S.ProofUpperBound}
This appendix is devoted to the proof of Proposition \ref{P7.UpperboundAuxPb}.

We prove the first assertion: when $M\geq d$ we have
\begin{equation*}
\inf_{v \in H^1_{g}(\O)} F_\v (v, \O) \leq \pi d b^2|\ln \xi|+\pi d |\ln\delta|+\mathcal{O}(1).
\end{equation*}
Fix first $d$ distinct points-centers of inclusions $a_1, ..., a_d$. Let $\rho_0:= 10^{-2}\cdot\min(\dist(a_i,\p\O),\min_{i \neq j} |a_i - a_j| > 0)$. Consider $\tilde v$ to be a smooth fixed function in $\W \setminus \cup_{i=1}^d \overline{B(a_i, {\rho_0})}$, such that $|\tilde v| =1$ in $\W \setminus \cup_{i=1}^d \overline{B(a_i, {\rho_0})}$ and
\[
\begin{cases}
\tilde v = g & \text{on } \d \O\\
\tilde v(x) = \frac{x-a_i}{|x-a_i|} &\text{on } \d B(a_i, \rho_0)
\end{cases}.
\]
Such a function clearly exists since the compatibility condition $\deg_{\p\O}(g)=\sum_{i=1}^d\deg_{ \d B(a_i, \rho_0)}(\tilde{v})$ is satisfied. Let $c_0 =10^{-2}\cdot \dist(0,\d \o) $.
For every $1 \leq i \leq M$, consider a disc  $B(a_i, c_0 \dl) \subset \w_\dl^i$. By the choice of $c_0$, we have $\dist(\d \w_\dl, B(a_i, c_0 \dl)) \geq c_0 \dl$. Therefore, using Proposition \ref{P7.UepsCloseToaeps}
\begin{equation}\label{aux_estimate}
U_\v^2 - b^2 \leq C e^{-\frac{c \dl}{\v}} \ \text{in} \ B(a_i, c_0 \dl).
\end{equation}
Consider the test function $v_0^\v$ defined as
\[
v_0^\v (x)=
\begin{cases}
\tilde v(x)&\text{for }x\in \O \setminus \cup_{i}\overline{B(a_i, \rho_0)}\\
\displaystyle\frac{x-a_i}{|x-a_i|}&\text{for }x \in B(a_i, \rho_0) \setminus \overline{B(a_i, \v)}\\
\displaystyle\frac{x - a_i}{\v}&\text{for } x \in B(a_i, \v)
\end{cases}.
\]
Using (\ref{aux_estimate}) and (\ref{7.MainHyp}) we have
\begin{eqnarray*}
\inf_{v \in H^1_{g}(\O)} F_\v (v, \O) &\leq& F_\v(v_0^\v)
\\& \leq& \pi d b^2|\ln \v|+\pi d (1-b^2)|\ln\delta|+C=\pi d b^2|\ln \xi|+\pi d |\ln\delta|+C.
\end{eqnarray*}
Now we prove the second assertion: when $M<d$ we have
\begin{equation*}
\inf_{v \in H^1_{g}(\O)} F_\v (v, \O) \leq \pi d b^2|\ln \xi|+\pi \sum_i d_i^2 |\ln\delta|+C.
\end{equation*}

 Let $d_1,...,d_M\in\N$ be s.t. $\sum d_i=d$. Set $c_0=10^{-2 d}\cdot \dist(0,\d \o)$. For $i\in\{1,...,M\}$ s.t. $d_i>0$, fix $\alpha_{1,i},...,\alpha_{d_i,i}\in B(0,10^{d}c_0)\subset\o$ s.t.
\[
\min\left(\min_{j\neq k}|\alpha_{j,i}-\alpha_{k,i}|,\dist(\alpha_{j,i},\p\o)\right)>4c_0.
\]
Consider an $\v$-dependent map $\tilde{v}_0^\v\in H^1(\O\setminus\cup_{d_i>0}\overline{B(a_i,10^{d}c_0\delta)},\S^1)$ s.t.
\[
\begin{cases}
\tilde v_0 ^\v= g & \text{on } \d \O\\
\tilde v_0^\v(x) = \dfrac{(x-a_i)^{d_i}}{|x-a_i|^{d_i}} &\text{on } \d B(a_i, 10^{d}c_0\delta)
\end{cases}
\]
and satisfying
\[
\int_{\O\setminus\cup_{d_i>0}\overline{B(a_i,10^{d}c_0\delta)}}|\n \tilde v_0^\v|^2\leq\pi\sum d_i^2|\ln\delta|+C
\]
with $C$ depending only on $\O,\o$ and $g$.

(Such maps do exist, \emph{e.g.}, consider the map introduced in \cite{BBH}, Remark I.5.)

For $i\in\{1,...,M\}$ s.t. $d_i>0$, we consider a map $v_i^\v\in H^1(B(0,10^{d}c_0)\setminus\cup_{j=1}^{d_i}\overline{B(\alpha_{j,i},\xi)},\S^1)$ s.t.
\begin{enumerate}[$\bullet$]
\item $v_i^\v(x)=x^{d_i}/|x|^{d_i}$ on $\p B(0,10^{d}c_0)$,
\item$v_i^\v(x)=(x-\alpha_{j,i})/|x-\alpha_{j,i}|$ on $\p B(\alpha_{j,i},\xi)$,
\item $\displaystyle\int_{B(0,10^{d}c_0)\setminus\cup_{j=1}^{d_i}\overline{B(\alpha_{j,i},\xi)}}|\n v_i^\v|^2\leq\pi d_i|\ln\xi|+C$ with $C$ depending only on $\o$.
\end{enumerate}
(For example, the map considered in Remark I.5 in \cite{BBH} has these properties).

The necessary test function that satisfies the bound \eqref{7.UpperboundAuxPbCaseII} is obtained by rescaling the $v_i^\v$'s (in order to have maps defined in balls of size $\delta$) and gluing the rescaled maps with $\tilde{v_0^\v}$.

\section{Proof of the $\eta$-ellipticity Lemma}\label{S.ProofEtaEllipticityLemma}

The main argument in the proof of the $\eta$-ellipticity result is the following convexity lemma which is a generalization of Lemma 8 in \cite{BeMi2}. The proof of Lemma \ref{L7.ConvexityLemma} is given in \cite{degzero}.
\begin{lem}\label{L7.ConvexityLemma}[Convexity Lemma]

Let $C$ be a chord in the closed unit disc, $C$ different from a diameter. Let $S$ be the smallest of two regions enclosed by the chord and the boundary of the disc.

Let $O$ be a Lipschitz, bounded, connected domain and let $g \in C(\d O, S)$.

Assume that $v$ minimizes Ginzburg-Landau type energy
\[
\tilde F (v) = \int_{O}\left\{ \tilde\a (x) |\nabla v|^2 +  \tilde\b (x)(1-|v|^2)^2 \right\}\, {\rm d}x
\]
in $H^1_g(O)$, with  $\tilde\a, \tilde\b\in L^\infty(O,\R)$ satisfying ${\rm essinf} \tilde\a >0,{\rm essinf} \tilde\beta > 0$. Then $v(O) \subset S$.
\end{lem}

We prove the first part of the lemma \ref{L7.BadDiscsLemma}. Let $x\in\O$ be s.t. $\dist(x,\p\O)\geq\v^{{1/4}}$. We have
\begin{eqnarray*}
F_{\ve}(v_\ve, B(x , \v^{1/4})) \geq  \frac{b}{2} \int_{B(x , \v^{1/4}) \setminus \overline{B(x , \v^{1/2})}} \left\{|\n v_\ve|^2 + \frac{1}{ \ve^2} (1 - |v_\ve|^2)^2\right\} \\
= \frac{b}{2} \int_{\v^{1/2}}^{\v^{1/4}} \frac{1}{r} \cdot r \int_{\d B(x,r)} \left\{ |\n v_\v|^2 + \frac{1}{\v^2} (1 - |v_\v|^2)^2\right\}.
\end{eqnarray*}
By Mean Value theorem, exists $r\in(\v^{1/2},\v^{1/4})$ s.t
\[
r\int_{\p B(x,r)}{\left\{|\n v_\v|^2+\frac{1}{\v^2}(1-|v_\v|^2)^2\right\}}\leq \frac{\frac{2}{b}F_\v(v_\v,B(x,\v^{1/4}))}{\frac{1}{4}|\ln\v|}.
\]
There is $C_2=C_2(\chi,b)>0$ s.t if $F_\v(v_\v,B(x,\v^{1/4}))\leq \chi^2|\ln\v|$, we have
\begin{equation}\label{7.BoundOnTheVarBisRepetita}
{\rm Var}(v_\v, \p B(x,r))\leq C_2\chi, \text{ where } \displaystyle{\rm Var}\,( v_\ve, \d B(x,r)):=\int_{\d B(x,r)}|\p_\tau v_\v|.
\end{equation}
It follows that
\begin{equation}\label{7.LowerBoundForTheModulusOnGoodDisc}
|v_\v|^2\geq1-3C_2\chi\text{ on }\p B(x,r).
\end{equation}
Indeed, arguing by contradiction, assume that there is $\v_n\downarrow0$ and $y_n\in\p B(x,r)$ s.t. $|v_{\v_n}(y_n)|^2<1-3C_2\chi$. Using \eqref{7.BoundOnTheVarBisRepetita} we obtain that
\[
|v_{\v_n}|^2\leq1-C_2\chi\text{ on }\p B(x,r)
\]
which implies that 
\begin{eqnarray*}
2\pi C_2^2\chi^2\v_n^{2(\frac{1}{2} - 1)}&\leq&\frac{2\pi C_2^2 r^2\chi^2}{\v_n^2}
\\&\leq&\frac{r}{\v_n^2}\int_{\p B(x,r)}{(1-|v_{\v_n}|^2)^2}
\\&\leq& \frac{\frac{2}{b}F_{\v_n}(v_{\v_n},B(x,{\v_n}^{1/4}))}{\frac{1}{4}|\ln\v_n|}\leq \frac{8 \chi^2}{b}.
\end{eqnarray*}
Clearly, the previous assertion gives contradiction.

From \eqref{7.BoundOnTheVarBisRepetita} and \eqref{7.LowerBoundForTheModulusOnGoodDisc}, there is $C=C(\chi,b)>0$ and $\v_0=\v_0(\chi)>0$ s.t.
for $\v<\v_0$,
\[
v_\v:\p B_r\to\{z\in B_1\,|\,\Re z>1-C\chi\}.
\]

Using Convexity Lemma (Lemma \ref{L7.ConvexityLemma}), we find that $|v_\v|\geq1-C\chi$ in $B(x,r)\supset B(x,\v^{1/2})$.

If $\dist(x,\p\O)<\v^{{1/4}}$, we denote $S_r=\O\cap \p B(x,r)$, $r\in(\v^{1/2},\v^{1/4})$. Clearly, we have
\[
\frac{2}{b}F_\v(v_\v,B(x,\v^{{1/4}})\setminus \overline{B(x,\v^{1/2})} )\geq\int_{\v^{1/2}}^{\v^{1/4}}\frac{1}{r}\cdot r\,{\rm d}r\int_{S_r}{\left\{|\n v_\v|^2+\frac{1}{\v^2}(1-|v_\v|^2)^2\right\}}.
\]
Using mean value argument and the facts that $g_\v\to g_0$ in $C^{1}(\p\O,\S^1)$ and that $0\leq1-|g_\v|\leq\v$, there are $r\in(\v^{1/2},\v^{1/4})$ and $C_1=C_1(\|g_0\|_{C^1},\O)$ s.t
\[
r\int_{\p (B(x,r)\cap\O)}{\left\{|\p_\tau v_\v|^2+\frac{1}{\v^2}(1-|v_\v|^2)^2\right\}}\leq \frac{\frac{2}{b}F_\v(v_\v,B(x,\v^{1/4}))+C_1}{\frac{1}{4}|\ln\v|}.
\]
Using the same argument as before (taking $O=\O\cap B(x,r)$) we obtain the desired result.

We prove the second part of the lemma. Let $\mu\in(0,1)$ and $x\in\{\dist(x,\p\O)\geq\v^{{1/4}}\}$. Using mean value argument, there is $r\in(\v^{1/2},\v^{1/4})$ s.t
\[
r\int_{\p B(x,r)}{\left\{|\n v_\v|^2+\frac{1}{\v^2}(1-|v_\v|^2)^2\right\}}\leq \frac{\frac{2}{b}F_\v(v_\v,B(x,\v^{1/4}))}{\frac{1}{4}|\ln\v|}.
\]
There exists $C_1=C_1(\mu,b)>0$ s.t if $F_\v(v_\v,B(x,\v^{1/4}))\leq C_1|\ln\v|$, we have
\[
{\rm Var}(v_\v, \p B(x,r))\leq \frac{1-\mu}{10}\text{ and }1-|v_\v|\leq\frac{1-\mu}{10}\text{ on }\p B(x,r).
\]
By Convexity Lemma $|v_\v|\geq\mu$ in $B(x,r)\supset B(x,\v^{1/2})$.

If $\dist(x,\p\O)<\v^{{1/4}}$, denote $S_r=\O\cap \p B(x,r)$, $r\in(\v^{1/2},\v^{1/4})$. Since
\[
\frac{2}{b} F_\v(v_\v,B(x,\v^{{1/4}})\setminus \overline{B(x,\v^{1/2})} )\geq\int_{\v^{1/2}}^{\v^{1/4}}\frac{1}{r}\cdot r\,{\rm d}r\int_{S_r}{\left\{|\n v_\v|^2+\frac{1}{\v^2}(1-|v_\v|^2)^2\right\}}
\]
and using the conditions on $g_\v$, by mean value argument there is $r\in(\v^{1/2},\v^{1/4})$ s.t
\[
r\int_{\p (B(x,r)\cap\O)}{\left\{|\p_\tau v_\v|^2+\frac{1}{\v^2}(1-|v_\v|^2)^2\right\}}\leq \frac{\frac{2}{b}F_\v(v_\v,B(x,\v^{1/4}))+C(\|g_0\|_{C^1},\O)}{\frac{1}{4}|\ln\v|}.
\]
Using the same argument as before, the statement of the lemma follows.\vspace{5mm}
\\\noindent{\bf Acknowledgement.} The authors thank Prof. Leonid Berlyand and Prof. Petru Mironescu for useful and stimulating discussions. The work of M. Dos Santos and O. Misiats
was partially supported by NSF grant DMS-0708324. Part of this work was done when M. Dos Santos was visiting Penn State University. He is grateful
to the Math Department and to Prof. Leonid Berlyand for the hospitality and support of his visit.

\bibliography{biblio}
\end{document}